\documentclass [preprint] {elsarticle} 

\usepackage{graphicx}
\usepackage{amssymb}
\usepackage{amsmath}
\usepackage[section]{placeins}
\usepackage{mathrsfs}  
\usepackage{multirow}
\usepackage{amsthm}
\usepackage{float}
\usepackage{bbm}
\usepackage{hyperref}
\usepackage{mwe}
\usepackage{subfig}
\usepackage{mathtools}
\usepackage{colortbl}

\def\RR{I\kern-0.35em R\kern0.2em \kern-0.2em}
\def\BBN{I\kern-0.35em  N\kern0.2em \kern-0.2em}

\newcommand{\be}{\begin{equation}}
\newcommand{\ee}{\end{equation}}
\newcommand{\bea}{\begin{eqnarray}}
\newcommand{\eea}{\end{eqnarray}}
\newcommand{\beann}{\begin{eqnarray*}}
\newcommand{\eeann}{\end{eqnarray*}}

\newtheorem{thm}{Theorem}[section]
\newtheorem{prop}[thm]{Proposition}
\newtheorem{lem}[thm]{Lemma}

\newtheorem{rem}[thm]{Remark}
\newtheorem{cor}[thm]{Corollary}

\usepackage{datetime} 
\usepackage[T1]{fontenc}
\usepackage[english] {babel}

\begin{document}
\begin{frontmatter}
\baselineskip=15pt

\title{Error analysis of a class of semi-discrete schemes for solving the Gross-Pitaevskii equation at low regularity}

\author[1]{Yvonne Alama Bronsard}

\address{LJLL (UMR 7598), Sorbonne Universit\'e}
\ead{yvonne.alama\_bronsard@sorbonne-universite.fr}


\begin{abstract} 
We analyze a class of time discretizations for solving the nonlinear Schr\"odinger equation with non-smooth potential and at low-regularity on an arbitrary Lipschitz domain $\Omega \subset \mathbb{R}^d$, $d \le 3$.
We show that these schemes, together with their optimal local error structure, allow for convergence under lower regularity assumptions on both the solution and the potential than is required by classical methods, such as splitting or exponential integrator methods. Moreover, we show first and second order convergence in the case of periodic boundary conditions, in any fractional positive Sobolev space $H^{r}$, $r \ge 0$, beyond the more typical $L^2$ or $H^\sigma (\sigma>\frac{d}{2}$) -error analysis. Numerical experiments illustrate our results.
\end{abstract}
\end{frontmatter}

\setcounter{tocdepth}{2}
\setcounter{secnumdepth}{4}
\section{Introduction}
We consider the Gross-Pitaevskii (GP) equation
\be \label{evGP}
i  \partial_t u(t,x) = - \Delta u(t,x) + V(x)u(t,x)  + \vert u(t,x)\vert^2 u(t,x),\quad (t,x) \in \mathbb{R}\times  \Omega 
\ee
with
$\Omega \subset \mathbb{R}^d$, $d \le 3$, and an initial condition
\begin{equation}
\label{init}
u_{|t=0}= u_{0}.
\end{equation}
When $\partial \Omega \neq \emptyset$, we assign boundary conditions which will be encoded in the choice of the domain of the operator $\mathcal{L} = i\Delta$. We recall that the linear operator $\mathcal{L} = i\Delta$ generates a group $\{e^{t\mathcal{L}}\}_{t \in \mathbb{R}}$ of unitary operators on $L^2(\Omega)$. We will deal with mild solutions of the initial value problem \eqref{evGP} and \eqref{init} which are given by Duhamel's formula;
\be \label{Duhamel}
u(t) = e^{i t \Delta } u_0 + \int_0^t e^{i(t-\zeta) \Delta} f(u,\bar{u}, V)(\zeta,x) d\zeta
\ee
where we denote the nonlinearity by 
\be\label{nonlin}
f(u,\bar{u}, V)(t,x) = -i(V(x)u(t,x)  + u^2(t,x)\bar{u}(t,x)).
\ee

Throughout this article we will be interested in studying numerical schemes which approximate the time dynamics of \eqref{evGP} at low-regularity, by means of appropriate approximations of Duhamel's formula. Namely, we are interested in providing a reliable approximation of \eqref{evGP} (or equivalently of \eqref{Duhamel}) when the initial data $u_0$ and the potential $V$ are non-smooth, in the sense that they belong to Sobolev spaces of low order. 

One setting for the Gross-Pitaevskii equation is to describe the dynamics of Bose-Einstein condensates in a potential trap. In many physically relevant situations the potential is assumed to be rough or disordered, and hence the study of equation \eqref{evGP} in this non-smooth or low-regularity framework is of physical interest (\cite{Disorder1}, \cite{physics2}).

Recently much progress has been made in the development of low-regularity approximations to nonlinear evolution equations. First, in the case of periodic boundary conditions a class of schemes called {\it Fourier integrators} \cite{ORS1} or {\it resonance based schemes} \cite{BS} were introduced to approximate the time dynamics of dispersive equations such as NLS, KdV, and Klein-Gordon (see \cite {HofS}, \cite{OS1} , \cite{CS_KG}). 
Recently, higher order extensions of these {\it resonance based schemes} were introduced in \cite{BS} for approximating in a unified fashion a large class of dispersive equations with periodic boundary conditions.
These {\it resonance based schemes} were shown to converge in a more general setting, namely under lower regularity assumptions, than classical methods required (see \cite{OS1}, \cite{ORS1} and references therein for a comparative analysis).
The name of these schemes is due to their construction which revolves around Fourier based expansions of the solution and of the resonant structure of the equation. We explain the idea behind these resonance based schemes in detail in Section \ref{periodic}.
These ideas were then extended in \cite{RS} to treat more general domains $\Omega \subset \mathbb{R}^d$ and boundary conditions and allow to deal with a class of parabolic, hyperbolic and dispersive equations in a unified fashion. The resulting schemes were termed {\it low-regularity integrators}, or {\it Duhamel's integrators} (see \cite{RS}). A next natural step in this study of low-regularity approximations to nonlinear PDEs is to introduce a potential term $uV$ with minimal regularity assumptions on the solution $u$ and the potential $V$. The first order low-regularity integrator for equation \eqref{evGP} was first stated in the report \cite{OberwolfachReport}, which also provides a preliminary discussion of previous results on low-regularity integrators for solving evolution equations. 
The goal of this article is to study first and second low-regularity schemes for equation \eqref{evGP}, with an emphasis on the error analysis. 
In a subsequent work \cite{ABBS} we present a general framework for deriving low-regularity schemes up to arbitrary order, using new techniques based on decorated trees series analysis to extend the construction of the schemes presented in this article.

In this article we study a class of low-regularity integrators to solve the Gross-Pitaevskii equation \eqref{evGP} on an arbitrary domain $\Omega \subset \mathbb{R}^d$.  
In the case where the domain is a torus $\mathbb{T}^d$,
we state and prove first and second order convergence in any fractional positive Sobolev space $H^{r}$, under moderate regularity assumptions on both the solution $u$ and the potential $V$. These are stronger convergence results than the more typical $L^2$ or $H^\sigma$ ($\sigma>\frac{d}{2}$) -convergence analysis, and apply to the nonlinear Schr\"odinger equation as an immediate consequence.
 In future work we will extend the Sobolev error estimates in the case of bounded smooth Lipschitz domains with homogeneous Dirichlet boundary conditions. We state our results in the next subsections, starting with the first order scheme and following with the second order scheme.
%
%
%
%
\subsection{First-order low regularity integrator}
In Section \ref{general1} we construct the following first order low-regularity integrator on $\Omega$, which was first stated in \cite{OberwolfachReport}.
For $n \ge 0$, we define,
\be\label{first-order}
u^{n+1} = \Phi^\tau_{num,1}(u^n) := e^{i\tau \Delta}[ u^n - i \tau(u^n \varphi_1(-i \tau \Delta ) V + (u^n)^2 \varphi_1(-2i \tau\Delta)\bar{u}^n )], \quad \text{where} \ u^0 = u_0,
\ee
and $\varphi_1(z) = \frac{e^z -1 }{z}$ is a bounded operator on $i\mathbb{R}$. 
The construction of this scheme does not rely on Fourier based techniques, 
%
and hence one can couple the above time discretization not only with spectral methods but with more general types of spatial discretizations. Indeed, on general domains one can call upon Krylov space methods for the approximation of the matrix exponential $e^{it\Delta}$, and the action of the $\varphi_1(\cdot)$ functions (see \cite{KIOPS} and \cite{ExpInt}). The fully discrete analysis on a smooth bounded domain with homogeneous Dirichlet conditions and with a finite elements space discretization, is the objective of future work.
%
Let us mention that a different construction of a low-regularity scheme for \eqref{evGP} based on tree series analysis is presented in the recent work \cite{ABBS} to obtain related low-regularity schemes.

We prove in Section \ref{Gerr1_section} the following first order convergence result for the scheme \eqref{first-order}, in the case of periodic boundary conditions. For the local-wellposedness result of \eqref{evGP} and \eqref{init} we refer to Theorem \ref{LWP} given in Section \ref{notation}.
\begin{thm}\label{Global1}
Let $T>0$, $r \ge 0$, 
and
\be\label{r_1}
r_1 := 
\begin{cases}
r + 1, \ \ \text{if} \ r > \frac{d}{2},\\
1+ \frac{d}{2} + \epsilon, \ \text{if} \ 0 < r \le \frac{d}{2},\\
1 + \frac{d}{4},  \ \ \text{if} \ r = 0,\\ 
\end{cases}
\ee
where $0 < \epsilon < \frac{1}{4}$ can be arbitrarily small. For every $u_0\in H^{r_1}(\mathbb{T}^d)$  and $V \in H^{r_1}(\mathbb{T}^d)$, let $u\in \mathcal{C}([0,T], H^{r_1}(\mathbb{T}^d))$ be the unique solution of \eqref{evGP}. Then there exists $\tau_{\min} > 0$ and $C_T > 0 $ such that for every time step size $\tau \le \tau_{\min}$ the numerical solution $u^n$ given in equation \eqref{first-order} has the following error bound:
\begin{align*} 
 \|u^n-u(n\tau)\|_{H^{r}}
\le C_T \tau, \quad 0 \le n\tau \le T.
\end{align*}
\end{thm}
Before moving on to the second order scheme and its convergence result we make a few remarks on the regularity assumptions made in the above theorem. 
A consequence of Theorem \ref{Global1} is that for any initial data and potential in $H^{r+1}(\Omega)$ where $r> \frac{d}{2}$ and $\Omega = \mathbb{T}^d$ (or on the full space $\Omega = \mathbb{R}^d$) we have the following global error estimate:
\begin{align*} 
\max_{1\le n\tau\le T} \|u^n-u(t_n)\|_{H^{r}} 
\le C(\sup_{[0,T]}||u(t) ||_{H^{r+1}}, ||V||_{H^{r+1}}) \tau.
\end{align*}
Namely we only ask one additional Sobolev derivative on the initial data $u_0$ and the potential $V$ in order to obtain first-order convergence of our low-regularity scheme \eqref{first-order}. This is due to the favorable local error structures that these low-regularity schemes inherit. See (\cite{OS1}, \cite{BS}, \cite{RS}), and references therein for an in depth comparative analysis of these low-regularity schemes with classical methods such as splitting methods, or exponential integrator methods.

Secondly, in the case $0<r\le\frac{d}{2}$ the convergence analysis in $H^r$-norm of a time discretization of equation \eqref{evGP} has not to our knowledge previously been studied, and these are the first convergence results in this regime.

Finally, when $r=0$, we consider the regularity assumptions required for an $L^2$-error analysis, and compare them with the existing $L^2$ convergence results for the Gross-Pitaevskii equation \eqref{evGP}. When $r = 0$, Theorem \ref{Global1} states,
\begin{equation*}
||u(t_{n}) - u^{n}||_{{L^2} }
\le C(\sup_{[0,T]}||u(t)||_{H^{{1 + \frac{d}{4}}}}, ||V||_{H^{{1 + \frac{d}{4}}}}) {\tau}, \quad 0 \le n\tau \le T.
\end{equation*}
To our knowledge, this is the first convergence result of this type with low-regularity assumptions on both the solution $u(t)$ and the potential $V$. Indeed, in the literature $L^2$-convergence results have been established for smooth potentials. See, for example \cite{LubichGP}, where the authors showed first-order convergence of a Lie splitting scheme for the linear Schr\"odinger equation with a potential term $uV$ where they require $V$ to be a $C^5$-smooth potential. 
The authors of \cite{Peterseim} were able to show first-order convergence to \eqref{evGP} of a Crank-Nicholson scheme for a rough, discontinuous potential $V$, which is of physical relevance in the context of Bose-Einstein condensates. Namely, for $V = V_{d} + V_{s} $, where $V_s \in C^{\infty}_0(\Omega)$ is smooth perturbation of a disordered potential $V_d \in L^{\infty}(\Omega)$, the authors obtained first order convergence of their scheme under -among other assumptions- $u_{t} \in L^2(0,T;H^2(\Omega))$, \cite[Theorem 4.1]{Peterseim}. Further, as detailed in \cite[Appendix A]{Peterseim}, due to the roughness of $V_d$ the highest regularity assumption one can hope for on the solution is $u(t) \in H^2(\Omega)$, and this regularity is required for their error analysis.
%
%
In contrast to these results Theorem \ref{Global1} permits low-regularity assumptions simultaneously on {\it both} $u(t)$ and $V$.
Finally, we refer to \cite[Corollary 20]{RS} where the authors show first order convergence in $L^2(\Omega)$ of a low-regularity scheme for the nonlinear Schr\"odinger equation (NLS) 
while analogously asking for $1+\frac{d}{4}$ Sobolev regularity on the initial data.
%
%
%
\subsection{Second-order low regularity integrator}
In [Section \ref{second}, Corollary \ref{cor:2scheme_FD}] we derive the following second order low-regularity integrator on $\Omega$.
For $n \ge 0$, we define,
\begin{align}\label{second-order}
u^{n+1} = \Phi^\tau_{num,2}(u^n)
&:= e^{i\tau \Delta}u^n - i \tau e^{i\tau \Delta}\left(u^n \varphi_1(-i \tau \Delta ) V + (u^n)^2 \varphi_1(-2i \tau\Delta)\bar{u}^n \right) \\\nonumber
&\quad -i\tau \left((e^{i\tau \Delta}u^n) \varphi_2(-i \tau \Delta ) (e^{i\tau \Delta}V) + (e^{i\tau \Delta}u^n)^2 \varphi_2(-2i\tau \Delta) e^{i\tau \Delta}\bar{u}^n \right)\\\nonumber
&\quad +i\tau e^{i\tau \Delta}\left(u^n \varphi_2(-i \tau \Delta ) V + (u^n)^2\varphi_2(-2i\tau \Delta)\bar{u}^n \right) \\ \nonumber
&\quad-\frac{\tau^2}{2} e^{i\tau \Delta}(|u^n|^4u^n + 3u^n |u^n|^2V - |u^n|^2u^n\bar{V} + u^nV^2 ),
\end{align}
where $\varphi_2(z) = \frac{e^z -\varphi_1(z)}{z}$ is a bounded operator on $i\mathbb{R}$. 
We present in [Section \ref{sec:stab2}, Corollary \ref{cor:stab2}] yet another derivation of a low-regularity second order scheme for \eqref{evGP}. We also offer in \cite{ABBS} a similar low-regularity scheme as above, using a different construction based on tree series analysis. As for the first order scheme, the above time discretization can be coupled with various spatial discretizations such as with finite elements.

We prove in Section \ref{Gerr2_section} the following second order convergence result for the scheme \eqref{second-order}, in the case of periodic boundary conditions.
\begin{thm}\label{Global2}
Let $T>0$, $r \ge 0 $, and 
\begin{equation}\label{r_2}
r_2 := 
\begin{cases}
r + 2, \ \ \text{if} \ r > \frac{d}{2},\\
2+ \frac{d}{2} + \epsilon, \ \text{if} \ 0 < r \le \frac{d}{2},\\
2 + \frac{d}{4},  \ \ \text{if} \ r = 0,\\ 
\end{cases}
\end{equation}
where $0 < \epsilon < \frac{1}{4}$ can be arbitrarily small. For every $u_0\in H^{r_2}(\mathbb{T}^d)$  and $V \in H^{r_2}(\mathbb{T}^d)$, let $u\in \mathcal{C}([0,T], H^{r_2}(\mathbb{T}^d))$ be the unique solution of \eqref{evGP}. Then there exists $\tau_{\min} > 0$ and $C_T > 0 $ such that for every time step size $\tau \le \tau_{\min}$ the numerical solution given in equation \eqref{second-order} has the following error bound:
\begin{align} 
\max_{1\le n\le N} \|u_{\tau}^n-u(n\tau)\|_{H^{r}}
\le C_T \tau^2.
\end{align}
\end{thm}

We comment on the regularity assumptions made in the above theorem. First, in the (smooth) regime $r> \frac{d}{2}$, we ask only for two additional derivatives on the initial data $u_0$ and the potential $V$. A preceding second order convergence result has been established for the NLS equation by \cite{OWY} in this regime.
Indeed, using a resonance-based approach the authors \cite{OWY} showed second order convergence in $H^r$, for $r> \frac{d}{2}$, and $u_0\in H^{r+2}$, of a low-regularity scheme for NLS (given by equation \eqref{second-order} with $V=0$). Secondly,
to our knowledge this is the first convergence result in $H^r$ of a second order time discretization of equation \eqref{evGP} in the intermediate regime $0<r\le\frac{d}{2}$.
Finally, we compare once again our result to the $L^2$-convergence results obtained in the literature; we mention the authors \cite{LubichGP} who show second order convergence of a Strang splitting scheme for a $C^5$-smooth potential. Whereas, the authors \cite{Peterseim} obtain second order convergence of a Crank-Nicholson scheme for a smooth potential $V$ and -among other assumptions- for $u_{tt} \in L^2(0,T;H^2(\Omega))$.
As mentioned previously, we emphasize  that in contrast to previous results we establish convergence under low-regularity assumptions on {\it both} $u$ and $V$.

\subsection{Outline of the paper}
We motivate the construction of the first order low-regularity integrator in Section \ref{periodic} by first deriving the scheme in the periodic setting. We then generalize to the construction of the low-regularity scheme for an arbitrary domain $\Omega \subset \mathbb{R}^d$. In Section \ref{second} we introduce the second order low-regularity integrator and discuss stability issues which arise. We then propose two different approaches to guarantee the stability of our proposed scheme. Section \ref{first} and \ref{second} also include the local and global error analysis of the first and second low-regularity integrators. Finally, in Section \ref{sec:numexp} we present numerical experiments underlining our theoretical findings. In the next section we briefly introduce some notation and nonlinear estimates which are crucial for the local and global error analysis.
\section{Notation and nonlinear estimates}\label{notation}
We begin by establishing some notation used in the paper, starting by the definition of the norm used throughout the error analysis sections.
%
%
%
Our analysis will be made in the periodic fractional Sobolev space, 
$$H^{r}(\mathbb{T}^d) := \{u = \sum_{k \in \mathbb{Z}^d}u_k \frac{e^{ikx}}{\sqrt{ (2\pi)^d}} \in L^2(\mathbb{T}^d): |u|_r^2 \triangleq \sum_{k \in \mathbb{Z}^d} |k|^{2 r}|u_k|^2 < +\infty \}$$
which is endowed with the norm
\begin{align*}
||u||_{H^r}^2 &= ||u||_{L^2(\mathbb{T}^d)}^2 + ||(-\Delta)^{r/2}u||_{L^2(\mathbb{T}^d)}^2\\
&= \sum_{k \in \mathbb{Z}^d} (1+ |k|^{2r})|u_k|^2,
\end{align*}
where $u_k = \frac{1}{\sqrt{ (2\pi)^d}}\int_{\mathbb{T}^d} u e^{-ikx}dx$.

Throughout the remainder of this section we fix $\sigma > \frac{d}{2}$, and we restrict the class of initial data and potential to belong to the Sobolev space $H^{\sigma}$.
We now present some nonlinear estimates which  will be fundamental for our analysis.
 We separate our $H^{r}$-error analysis into three cases: $r = 0$, $0<r \le \frac{d}{2}$, and $r > \frac{d}{2}$.
First, when $r = 0$, using the Sobolev embedding $H^{\sigma} \hookrightarrow L^{\infty}$, we have the following nonlinear $L^2$ estimate:
\be\label{L^2-bilin}
||vw||_{L^2} \lesssim ||v||_{H^{\sigma}} ||w||_{L^2}.
\ee
%

In the case where $0<r\le \frac{d}{2}$ we have, 
\be\label{bilin-nonsmooth}
||vw||_{H^{r}} \lesssim ||v||_{H^{\sigma}} ||w||_{H^{r}},
\ee
while in the regime $r > \frac{d}{2}$ the above holds with $\sigma = r$ (see for example \cite[equation (2.49)]{HerrSchratz}). For completeness, we provide a proof of the above inequality \eqref{bilin-nonsmooth} in the Appendix for both regimes of $r$. These estimates will be used frequently throughout the error analysis sections (see Sections \ref{sec:LocalS1}, \ref{Gerr1_section}, \ref{sec:localO2}, \ref{Gerr2_section}).

One can easily deduce from the inequalities \eqref{L^2-bilin} and \eqref{bilin-nonsmooth} the following estimates on the nonlinearity \eqref{nonlin};
\be\label{A2.2}
\begin{split}
||f(w,\bar{w}, V)|| _{H^{r}} 
&\le C_{r,\sigma}(||w||_{H^{\sigma}}, ||V||_{H^{\sigma}}) ||w||_{H^{r}} \\
||f(w,\bar{w}, V) - f(v, \bar{v}, V)||_{H^{r}} 
&\le C_{r,\sigma}(||w||_{H^{\sigma}}, ||v||_{H^{\sigma}}, ||V||_{H^{\sigma}})||w-v ||_{H^{r}},
\end{split}
\ee
where $C_{r,\sigma}(||u||,||v||,||w||)$ denotes a generic constant which depends on the bounded arguments $||u||$, $||v||$, and $||w||$. In the regime $r > \frac{d}{2}$ the above holds with $\sigma = r$.
\begin{rem}
In order to deal with less smooth initial data $u_0 \in H^r, r \le \frac{d}{2}$ one cannot make use of the bilinear estimates \eqref{L^2-bilin}, \eqref{bilin-nonsmooth} and one would need to call upon more subtle tools to show appropriate fractional convergence of the scheme. Several works have been made when working on $\Omega = \mathbb{T}$ or $\Omega = \mathbb{R}^d$ where low-regularity estimates for very rough data $u_0 \in H^r, r \le \frac{d}{2}$ could be obtained by using tools from dispersive PDE such as discrete Strichartz estimates, or Bourgain spaces, see \cite{ORS1}, \cite{ORS2}. This refined error analysis is out of scope for this paper.
\end{rem}

We finish this subsection by stating the following local well-posedness result of a solution to \eqref{evGP} and \eqref{init} of the form \eqref{Duhamel}. Indeed, using the estimates \eqref{A2.2}, one obtains from a classical Banach fixed point argument the following result: 
%
\begin{thm}\label{LWP}
Let $\sigma_0 > \frac{d}{2}$. Given any $u_0\in H^{\sigma_0}(\mathbb{T}^d)$, and $V\in H^{\sigma_0}(\mathbb{T}^d)$ there exists $T >0$ and a unique solution $u\in C([0,T], H^{\sigma_0}(\mathbb{T}^d))$ to \eqref{evGP}. 
\end{thm}

\section{First order scheme and analysis}\label{first}
%
In this section we start by giving the main ideas behind the construction of the first order low-regularity scheme (see \cite{RS}).
We propose a novel low-regularity integrator for the approximation of Duhamel's formula \eqref{Duhamel}.
%
%
We will approximate equation \eqref{Duhamel} at the time step $t_n + \tau$, where $\tau$ is the time step size. By iterating Duhamel's formula \eqref{Duhamel}, we obtain the first order iteration 
\be \label{iter}
u(t_n + \tau) = e^{i\tau \Delta}[ u(t_n) - i \mathcal{J}_1(\tau, \Delta, u(t_n))] + R_{1,0}(\tau,u)
\ee
where the principal oscillatory integral (at first order) is given by
\be \label{PO}
\mathcal{J}_1(\tau, \Delta, v)=
\int_0^\tau e^{-i\zeta \Delta} [V(x)(e^{i\zeta \Delta}v) + (e^{i\zeta \Delta}v)^2(e^{-i\zeta \Delta}\bar{v})] d\zeta
\ee
and the remainder
\begin{equation*}
R_{1,0}(\tau,u) = \int_0^\tau e^{i(\tau - \zeta)\Delta }[f(u(t_n + \zeta), \bar{u}(t_n + \zeta), V) -f(e^{i\zeta \Delta}u(t_n), e^{-i\zeta \Delta}\bar{u}(t_n), V)] d\zeta.
\end{equation*}
We will construct a suitable discretization of the integral \eqref{PO} to allow for a low-regularity approximation to the first order Duhamel iterate \eqref{iter}. The idea is to filter out the dominant parts, which we denote by $\mathcal{L}_{dom}$, of the nonlinear frequency interactions within the integral \eqref{PO} and embed them in the discretization. The lower-order parts will be approximated and incorporated in the local error analysis.

First, to illustrate the underlying idea and to provide intuition behind the construction of these low-regularity integrators 
we start by analyzing the case of periodic boundary conditions $\Omega = \mathbb{T}$, with $V$ a periodic potential. The ideas presented in the next section were first introduced by the authors \cite{OS1} for solving a class of semilinear Schr\"odinger equations. After presenting the periodic case in a formal way, we rigorously detail in Section \ref{general1} the construction of the first order scheme in the more general case of an arbitrary domain $\Omega \subset \mathbb{R}^d$.
\subsection{Case of periodic boundary conditions: $\Omega = \mathbb{T}$}\label{periodic}
Assuming that $v \in L^2$, we can expand $v$ in Fourier series $v = \sum_{k \in \mathbb{Z}} \hat{v}_k e^{ikx}$. This allows us to express the action of the Schr\"odinger flow on v, $e^{\pm it \Delta } v(x) = \sum_{k \in \mathbb{Z}} \hat{v}_k e^{\mp it k^2 }e^{ikx}$.
Similarly assuming $V \in L^2$ we have $V(x) = \sum_{l \in \mathbb{Z}} \hat{V}_l e^{ilx}$.
In Fourier space, the oscillatory integral \eqref{PO} is then given by,
\begin{equation}\label{PO2}
\mathcal{J}_1(\tau, \Delta, v)= \sum_{l_1+ l_2 = l } \hat{V}_{l_1} \hat{v}_{l_2}e^{ilx} \int_0^\tau e^ {i\zeta R_2(l)} d\zeta + \sum_{-k_1+k_2+ k_3 =k } \bar{\hat{v}}_{k_1} \hat{v}_{k_2} \hat{v}_{k_3} e^{ikx} \int_0^\tau e^{i \zeta R_1(k) }d\zeta 
\end{equation}
with the resonance structure,
\be\label{res}
R_1(k) = 2k_1^2 -2k_1(k_2+k_3) + 2k_2k_3, \ \text{and} \ \ R_2 (l) = l_{1}^2 + 2l_1l_2.
\ee
Ideally we would like to integrate all the nonlinear frequency interactions \eqref{res} exactly and embed them in the discretization. This, however, would result in a generalized convolution (of Coifman-Meyer type \cite{CM}), which cannot be rewritten in physical space. Hence, the computations would need to be fully made in Fourier space. Carrying this out in higher spatial dimensions $d$ would cause large memory and computational efforts of order $O(K^{d \cdot \ell})$, where $K$ denotes the highest frequency in the discretization and $\ell$ is the number of factors in the nonlinearity. For practical computations, we want to be able to express the discretization also in physical space in order to use the Fast Fourier Transform (FFT) which is of computational effort of order $O(|K|^d log|K|^d)$. Therefore, we choose in the following an approximation of the integral \eqref{PO2} which allows for a practical implementation (by not performing exact integration), while optimizing the local error in the sense of regularity. We detail this procedure below.

%
We can extract the dominant and lower-order parts from the resonance structures \eqref{res} by recalling that $2k_1^2$ and $l_{1}^2$ correspond to second order derivatives in Fourier while the terms $k_m k_j $ (for $m \not= j$) correspond to product of first order derivatives.
We choose, 
\begin{equation*}
R_1(k) = \mathcal{L}_{dom,1}(k_1) + \mathcal{L}_{low,1}(k_1, k_2, k_3), \ \ R_2(l) = \mathcal{L}_{dom,2}(l_1) + \mathcal{L}_{low,2}(l_1, l_2)
\end{equation*}
with
\begin{align*}
&\mathcal{L}_{dom,1}(k_1) = 2k_1^2, \ \mathcal{L}_{low,1}(k_1, k_2, k_3) =  -2k_1(k_2+k_3) + 2k_2k_3, \ \text{and} \\ \\
&\mathcal{L}_{dom,2}(l_1) = l_{1}^2, \ \ \mathcal{L}_{low,2}(l_1, l_2) = 2l_1 l_2.
\end{align*}
From the above and from equation \eqref{PO2}, by a simple Taylor's expansion on the lower-order parts we (formally) allow for the following approximation of the oscillatory integral in Fourier space,
\begin{align*}
\mathcal{J}_1(\tau, \Delta, v) &= \sum_{l_1+ l_2 = l } \hat{V}_{l_1} \hat{v}_{l_2}e^{ilx} \int_0^\tau e^{i\zeta \mathcal{L}_{dom,2}(l_1)} e^{i\zeta \mathcal{L}_{low,2}(l_1, l_2)} d\zeta \\
&\qquad + \sum_{-k_1+k_2+ k_3 =k } \bar{\hat{v}}_{k_1} \hat{v}_{k_2} \hat{v}_{k_3} e^{ikx} \int_0^\tau e^{i \zeta \mathcal{L}_{dom,1}(k_1)} e^{i \zeta\mathcal{L}_{low,1}(k_1, k_2, k_3) }d\zeta \\
& = \sum_{l_1+ l_2 = l } \hat{V}_{l_1} \hat{v}_{l_2}e^{ilx} \int_0^\tau e^{i\zeta \mathcal{L}_{dom,2}(l_1)} \big(1+ O(\zeta \mathcal{L}_{low,2}(l_1, l_2))\big) d\zeta \\
&\qquad + \sum_{-k_1+k_2+ k_3 =k } \bar{\hat{v}}_{k_1} \hat{v}_{k_2} \hat{v}_{k_3} e^{ikx} \int_0^\tau e^{i \zeta \mathcal{L}_{dom,1}(k_1)} \big(1+ O(\zeta\mathcal{L}_{low,1}(k_1, k_2, k_3))\big) d\zeta. \\
\end{align*}
Mapping back into physical space we thus have 
\begin{align*}
&\mathcal{L}_{dom,1}(v) = -2\Delta v, \ \  \mathcal{L}_{low,1}(v) = 2(2{ |\nabla v|}^2v - |\nabla v|^2\bar{v}),\\
&\mathcal{L}_{dom,2}(v) = -\Delta v,  \ \ \mathcal{L}_{low,2}(v, V) = -2\nabla V \nabla v,
\end{align*}
and
\begin{align} \label{discr_OI}
\mathcal{J}_1(\tau, \Delta, v)
&= \int_0^\tau [e^{i\zeta \mathcal{L}_{dom,2}} V]v + [e^{i\zeta \mathcal{L}_{dom,1}} \bar{v}]v^2 + O\left(\zeta (\mathcal{L}_{low,2}(v, V) + \mathcal{L}_{low,1}(v))\right)d\zeta \\ \nonumber
&= \tau[v \varphi_1(i \tau \mathcal{L}_{dom,2}) V + v^2 \varphi_1(i \tau \mathcal{L}_{dom,1})\bar{v} ] + O\left(\tau^2 (\mathcal{L}_{low,2}(v,V) + \mathcal{L}_{low,1}(v))\right).
\end{align}
Hence, for a small time step $\tau$, by plugging the above expression of $\mathcal{J}_1$ in the iterate \eqref{iter} and ignoring the lower-order terms yields the first-order resonance based discretization
\be\label{scheme_periodic}
u^{n+1} = e^{i\tau \Delta}[ u^n - i \tau(u^n \varphi_1(-i \tau \Delta ) V + (u^n)^2 \varphi_1(-2i \tau\Delta)\bar{u}^n )].
\ee
The above scheme \eqref{scheme_periodic} has a favorable local error structure; namely from equation \eqref{discr_OI} we see that (formally) this discretization only ask for {\it first order} derivatives on the initial data and potential.

We now place ourselves in the general framework $\Omega \subset \mathbb{R}^d$, and make use of filtering techniques to recover the first order low-regularity approximation \eqref{scheme_periodic} in this general setting. The ideas presented in the next section are inspired by the work of \cite{RS}.
%
%
\subsection{General boundary conditions: $\Omega \subset \mathbb{R}^d$}\label{general1}
%
The goal of this section is to construct a first order discretization of the oscillatory integral \eqref{PO} when working on a general domain $\Omega$, and which allows for the improved local error structure \eqref{discr_OI} established in the preceding section. This is achieved by introducing a properly chosen filtered function which will filter out the dominant oscillatory terms $\mathcal{L}_{dom,1}, \mathcal{L}_{dom,2}$ explicitly found in the preceding section.

First, we recall the definition of the commutator-type term $\mathcal{C}[H,L]$ for $H(v_1, \cdots, v_n), \ n \ge 1$, a function and $L$ a linear operator:
$$
\mathcal{C}[H,L](v_1, \cdots, v_n) = -L(H(v_1, \cdots, v_n)) + \sum_{i=1}^{n}D_i H(v_1, \cdots, v_n) \cdot Lv_i.
$$
We make an important note that the above differs from the well known Lie commutator used for the error analysis of classical methods, such as for splitting methods (see for example \cite{LW}, \cite{LubichGP}).

We define the filtered function by
\be\label{filterfcn}
\begin{split}
\mathcal{N}(\tau, s, \zeta, \Delta, v) 
&= e^{-is\Delta} [e^{is\Delta}e^{-i\zeta \Delta}V(e^{is\Delta}v) + (e^{is\Delta} v)^2(e^{is\Delta}e^{-2i\zeta \Delta} \bar{v})].
\end{split}
\ee
The principal oscillations \eqref{PO} can be expressed with the aid of the filter function $\mathcal{N}$ as
\begin{equation*}
\mathcal{J}_1(\tau, \Delta, v)= \int_0^\tau \mathcal{N}(\tau, \zeta, \zeta, \Delta, v) d\zeta.
\end{equation*}
By the fundamental theorem of calculus we have
\be\label{Taylor1st}
\begin{split} 
\mathcal{J}_1(\tau, \Delta, v)
&= \int_0^\tau \mathcal{N}(\tau, 0, \zeta,v)d\zeta + \int_0^\tau \int_0^\zeta \partial_s \mathcal{N}(\tau,s,\zeta,v)ds d\zeta \\
\end{split}
\ee
where 
\be\label{N0}
\mathcal{N}(\tau, 0, \zeta,v) =
[e^{i\zeta \mathcal{L}_{dom,2}} V]v + [e^{i\zeta \mathcal{L}_{dom,1}} \bar{v}]v^2
\ee
and
\begin{equation*}
\partial_s \mathcal{N}(\tau ,s,\zeta,v) = e^{-is\Delta} \mathcal{C}[f,i\Delta](e^{is\Delta}v, e^{is\Delta}e^{-2i\zeta\Delta}\bar{v},  e^{is\Delta}e^{-i\zeta\Delta}V), 
\end{equation*}
\be\label{com}
\mathcal{C}[f,i\Delta](u,v,w) = -2i(\nabla w \cdot \nabla u + |\nabla u|^2 v + \nabla (u^2) \cdot \nabla v).
\ee
Hence, we recover the discretization of the oscillatory integral \eqref{PO} together with an improved local error structure of the form \eqref{discr_OI};
$$
\mathcal{J}_1(\tau, \Delta, v) = \tau[v \varphi_1(i \tau \mathcal{L}_{dom,2}) V + v^2 \varphi_1(i \tau \mathcal{L}_{dom,1})\bar{v} ] + R_{1,1}(\tau)
$$
where 
$$
R_{1,1}(\tau) = \int_0^\tau \int_0^\zeta e^{-is\Delta} \mathcal{C}[f,i\Delta](e^{is\Delta}v, e^{is\Delta}e^{-2i\zeta\Delta}\bar{v},  e^{is\Delta}e^{-i\zeta\Delta}V) ds d\zeta.
$$
\begin{cor} \label{first-order-cor}
The exact solution $u$ of \eqref{evGP} can be expanded as
$$
u(t_n + \tau) = e^{i\tau \Delta}[ u(t_n) - i \tau(u(t_n) \varphi_1(-i \tau \Delta ) V + (u(t_n))^2 \varphi_1(-2i \tau\Delta)\bar{u}(t_n) )] + \mathcal{R}_1(\tau, t_n)
$$
where the remainder is given by
\begin{equation}\label{R1}
\begin{split}
\mathcal{R}_1(\tau, t_n)
&= \int_0^\tau e^{i(\tau - \zeta)\Delta }[f(u(t_n + \zeta), \bar{u}(t_n + \zeta), V) -f(e^{i\zeta \Delta}u(t_n), e^{-i\zeta \Delta}\bar{u}(t_n), V)] d\zeta \\
& + \int_0^\tau \int_0^\zeta e^{i(\tau - s)\Delta} \mathcal{C}[f,i\Delta](e^{is\Delta}u(t_n), e^{is\Delta}e^{-2i\zeta\Delta}\bar{u}(t_n),  e^{is\Delta}e^{-i\zeta\Delta}V) ds d\zeta.
\end{split}
\end{equation}
\end{cor}

The first order low-regularity scheme \eqref{first-order} follows from the above Corollary \ref{first-order-cor} by neglecting the remainder $\mathcal{R}_1(\tau, t_n)$. We next show the first-order error estimates for the scheme \eqref{first-order} by first estimating it's favorable commutator-type local error structure.
\subsection{Local error estimates}\label{sec:LocalS1}
%
\begin{prop}\label{LocalS1}
Let $T > 0$, $r \ge 0$, and $r_1$ as in Theorem \ref{Global1}, namely
\begin{equation*}
r_1 = 
\begin{cases}
r + 1, \ \text{if} \ r > \frac{d}{2},\\
1+ \frac{d}{2} + \epsilon, \ \text{if} \ 0 < r \le \frac{d}{2},\\
1 + \frac{d}{4},  \ \text{if} \ r = 0,\\ 
\end{cases}
\end{equation*}
where $0 < \epsilon < \frac{1}{4}$ can be arbitrarily small.
%
%
%
Assume there exists $C_T>0$ such that
\be\label{reg1u}
\sup_{[0,T]}||u(t)||_{H^{r_1}} \le C_T, \ \text{and} \ \ ||V||_{H^{r_1}} \le C_T,
\ee
then there exists $M_T >0$ 
such that for every $\tau \in (0,1]$, 
\be\label{local-r_0}
||\mathcal{R}_1(\tau, t_n)||_{H^{r}} \le M_T \tau^2, \qquad 0 \le t_n \le T,
\ee
where $t_n = n\tau$ and $\mathcal{R}_1(\tau, t_n)$ is given in equation \eqref{R1}.
\end{prop}
\begin{proof} 
We write the error term $\mathcal{R}_1(\tau, t_n)$, defined in equation \eqref{R1}, as the sum of two terms, $\mathcal{R}_1(\tau, t_n) = \mathcal{G}^1(\tau,t_n) + \mathcal{G}^2(\tau,t_n)$. We begin by estimating the first term $\mathcal{G}^1(\tau,t_n)$.
Using the inequalities \eqref{A2.2} on $f$, and the boundedness of $e^{it\Delta}$ on Sobolev spaces we have that for all $\sigma > \frac{d}{2}$,
\be\label{I_o1}
\begin{split}
|| \mathcal{G}^1(\tau,t_n) ||_{H^{r}} 
&\le \tau \sup_{\zeta \in [0, \tau]}||f(u(t_n + \zeta), \bar{u}(t_n + \zeta), V) -f(e^{i\zeta \Delta}u(t_n), e^{-i\zeta \Delta}\bar{u}(t_n), V)||_{H^{r}}\\
& \le \tau C_{r,\sigma}(\sup_{[0,T]} ||u(t)||_{H^\sigma}, ||V||_{H^{\sigma}}) \sup_{\zeta \in [0, \tau]}||u(t_n + \zeta) - e^{i\zeta \Delta}u(t_n)||_{H^r} \\
&\le \tau C_{r,\sigma}(\sup_{[0,T]} ||u(t)||_{H^\sigma}, ||V||_{H^{\sigma}}) \sup_{\zeta \in[0, \tau]}||\int_0^\zeta e^{i(\zeta-s)\Delta}f(u(t_n+s), \bar{u}(t_n + s), V)ds||_{H^{r}}\\
&\le \tau^2C_{r,\sigma}(\sup_{[0,T]} ||u(t)||_{H^\sigma}, ||V||_{H^{\sigma}})\sup_{s\in [0, \tau]}||f(u(t_n+s), \bar{u}(t_n + s), V) ||_{H^{r}}\\
&\le C_{r,\sigma}(\sup_{[0,T]} ||u(t)||_{H^\sigma}, \sup_{[0,T]} ||u(t)||_{H^{r}}, ||V||_{H^{\sigma}})\tau^2,
\end{split}
\ee
where we use Duhamel's formula to go from the third to the fourth line.

We first note that by definition \eqref{r_1} of $r_1$ we clearly have that $r_1 > r$, and hence, 
$$
\sup_{[0,T]} ||u(t)||_{H^{r}} \le C_T.
$$
In the regime $r > \frac{d}{2}$, we take $\sigma = r$ in equation \eqref{I_o1}, which by the above remark clearly yields the desired bound $|| \mathcal{G}^1(\tau,t_n) ||_{H^{r}} \le C_T\tau^2$.

When $r \le \frac{d}{2}$, we will construct an appropriate $\sigma > \frac{d}{2}$ which will be used throughout the remainder of the proof when making the analysis in this non-smooth regime.\\
Let $0< \epsilon < \frac{1}{4}$, and let
\begin{equation}\label{sigma_0}
\sigma_0 = \frac{d}{2} + \frac{\epsilon}{2}.
\end{equation}
For $d\le 3$, we have that $\frac{d}{2} + \epsilon < 1+ \frac{d}{4}$, and hence $\sigma_0$ satisfies 
\begin{equation}\label{sig_0_b}
\frac{d}{2} < \sigma_0 < \frac{d}{2} + \epsilon < 1+ \frac{d}{4}.
\end{equation}
Consequently, by recalling the definition \eqref{r_1} of $r_1$, we have that $r_1 > \sigma_0$ when $r \le \frac{d}{2}$. Moreover, in the regime $0< r\le \frac{d}{2}$ we have the better bound $r_1 > \sigma_0 + 1$.\\
Hence, in the regime $r \le \frac{d}{2}$, since $r_1 > \sigma_0$ we obtain the desired bound $|| \mathcal{G}^1(\tau,t_n) ||_{H^{r}} \le C_T\tau^2$ by taking $\sigma = \sigma_0$ in equation \eqref{I_o1}.

We now estimate the second term $\mathcal{G}^2(\tau,t_n)$ in the remainder \eqref{R1}. From the explicit expression of the commutator \eqref{com}  and by making use of the nonlinear estimate \eqref{bilin-nonsmooth} we have for all $\sigma > \frac{d}{2}$,
\begin{align}\label{com-estimate1}
||\mathcal{C}[f,i\Delta](u,v,w)||_{H^{r}} &\le C_{r}( ||\nabla w \cdot \nabla u ||_{H^{r}} + |||\nabla u|^2 v||_{H^{r}} + 2||u \nabla u\cdot \nabla v||_{H^{r}}) \\ \nonumber
&\le C_{r}(||\nabla w||_{H^{\sigma}} ||\nabla u||_{H^{r}} + ||v||_{H^{\sigma}} ||\nabla u||_{H^{\sigma}}||\nabla u||_{H^{r}} + 2 ||u||_{H^{\sigma}}||\nabla u||_{H^{\sigma}} ||\nabla v||_{H^{r}})\\ \nonumber
&\le C_{r}(||u||_{r+1}, ||v||_{r+1}, ||u||_{\sigma+1}, ||v||_{\sigma+1}, ||w||_{\sigma+1}).
\end{align}

In the regime $r > \frac{d}{2}$, we have $r_1 = r + 1$ and by taking $\sigma = r$ in the above expression it follows that,
\begin{equation*}
\begin{split}
||\mathcal{C}[f,i\Delta](u,v,w)||_{H^{r}} &\le C_{r}(||u||_{r+1},||v||_{r+1},||w||_{r+1})\\
& \le C_{r}(||u||_{r_1},||v||_{r_1},||w||_{r_1}).
\end{split}
\end{equation*}

When $0<r \le \frac{d}{2}$, we take $\sigma=\sigma_0$ in the expression \eqref{com-estimate1}, where $\sigma_0$ is defined at \eqref{sigma_0}. Using the fact that $\sigma_0>\frac{d}{2} \ge r$ and $r_1 > \sigma_0 +1$ when $0<r \le \frac{d}{2}$, we obtain
\begin{equation*}
\begin{split}
||\mathcal{C}[f,i\Delta](u,v,w)||_{H^{r}} 
&\le C_{r}(||u||_{r+1}, ||v||_{r+1}, ||u||_{\sigma_0+1},||v||_{\sigma_0+1},||w||_{\sigma_0+1})\\
& \le C_{r}(||u||_{\sigma_0+1}, ||v||_{\sigma_0+1} ,||w||_{\sigma_0+1})\\
& \le C_{r}(||u||_{r_1},||v||_{r_1},||w||_{r_1}).
\end{split}
\end{equation*}
Finally we are left to treat the case $r = 0$, namely the error analysis in the $L^2$-norm. We wish to obtain more favorable regularity assumptions on $u(t)$ and $V$ (better than $\frac{d}{2} + 1 + \epsilon$) which are sufficient to bound the commutator term $\mathcal{G}^2$ in the $L^2$-norm. Instead of using the nonlinear estimate \eqref{bilin-nonsmooth} used to derive equation \eqref{com-estimate1}, we will make use of the Sobolev embedding $H^{\frac{d}{4}} \hookrightarrow L^4$ together with the embedding $H^{\sigma_0} \hookrightarrow L^{\infty}$, where $\sigma_0$ is defined in equation \eqref{sigma_0}. By recalling from equation \eqref{sig_0_b} that $\sigma_0 < 1+\frac{d}{4}$, we obtain
\begin{align*}
||\mathcal{C}[f,i\Delta](u,v,w)||_{L^2} &\le C( ||\nabla w \cdot \nabla u ||_{L^2} + |||\nabla u|^2 v||_{L^2} + 2||u \nabla u\cdot \nabla v||_{L^2}) \\
&\le C(||\nabla w||_{L^4} ||\nabla u||_{L^4} + ||v||_{H^{\sigma_0}} ||\nabla u||_{L^4}||\nabla u||_{L^4} + 2 ||u||_{H^{\sigma_0}}||\nabla u||_{L^4} ||\nabla v||_{L^4})\\
&\le C(||u||_{H^{1+\frac{d}{4}}},||v||_{H^{1+\frac{d}{4}}},||w||_{H^{1+\frac{d}{4}}}).
\end{align*}
Hence, by definition \eqref{r_1} of $r_1$, given any $r \ge 0$ we have shown the following bound,
\begin{equation}\label{com_bound}
||\mathcal{C}[f,i\Delta](u,v,w)||_{H^{r}} \le C_{r}(||u||_{r_1},||v||_{r_1},||w||_{r_1}).
\end{equation}
%
%
%
%
 %
Further, since $e^{is\Delta}$ is an isometry on Sobolev spaces we obtain the following estimate of $\mathcal{G}^2$ in $H^{r}$ norm,
$$
||\mathcal{G}^2(\tau,t_n)||_{H^{r}} \le C_{r}(\sup_{[0,T]}||u(t)||_{H^{r_1}},||V||_{H^{r_1}})\tau^2
$$
where $r_1$ is defined in equation \eqref{r_1}.
The local error estimate is hence demonstrated.
\end{proof}
We finish this section by making two remarks on the above proof.
\begin{rem}
We use the Sobolev embedding $H^{\frac{d}{4}} \hookrightarrow L^4$ for the analysis in the $L^2$-norm, as it leads to the more optimal regularity assumption $H^{1+ \frac{d}{4}}$ on the data and potential. Using this approach for an analysis in a higher Sobolev norm $H^{r}$ with $0< r \le \frac{d}{2}$, does not necessarily yield a better regularity assumption than $1+ \frac{d}{2} + \epsilon$. For example, in dimensions $d=2$ or $d=3$, an analysis made in the $H^1$-norm and  using the above Sobolev embedding would ask for $H^{2+ \frac{d}{4}}$ regularity, however we have the strict inequality $2+ \frac{d}{4} > 1+ \frac{d}{2} +\epsilon$. 
\end{rem}
\begin{rem}
By following the proof of Proposition \ref{LocalS1} one can ask for less Sobolev regularity on the potential $V$, while asking for more regularity on the solution $u(t)$. Indeed, for example, by making the analysis in the $L^2$-norm, one can ask for $V\in H^1$, and $u(t) \in 1+\frac{d}{2} + \epsilon$.
\end{rem}
\subsection{Global error estimates}\label{Gerr1_section}
Using the local error estimates established in the preceding section we show global first order convergence of our scheme \eqref{first-order} under the favorable regularity assumptions on the initial condition and the potential established previously.
\begin{proof}[Proof of Theorem \ref{Global1}]
Let $e^n= u^n - u(t_n)$, where $u^n= \Phi_{num,1}^{\tau}(u^{n-1})$ is given in equation \eqref{first-order}.
We begin by decomposing the error term as follows,
\be\label{err-decomp}
||e^{n+1}||_{H^{r}} = ||\Phi_{num,1}^\tau(u(t_n))-u(t_{n+1})||_{H^{r}} + ||\Phi_{num,1}^\tau(u^n) -\Phi_{num,1}^\tau(u(t_n)) ||_{H^{r}} . 
\ee
The first term of the above expression is given by the local error $\mathcal{R}_1(\tau, t_n)$ defined at equation \eqref{R1}, and which is of order $\tau^2$ by Proposition \ref{LocalS1}.
We wish to establish a stability estimate of the numerical flow $\Phi_{num,1}^\tau$ to bound the second term in equation \eqref{err-decomp}, and to conclude by a {\it Lady Windermere's fan} argument (\cite{LW}).
%

By using the estimates \eqref{L^2-bilin} and \eqref{bilin-nonsmooth}, together with the fact that $e^{i\xi\Delta}$ and $\varphi_1(i\xi\Delta)$ are bounded on Sobolev spaces (for all $\xi \in \mathbb{R}$), it easily follows from the definition of our scheme \eqref{first-order} that for all $r \ge 0$ and $\sigma > \frac{d}{2}$,
\be\label{stab-estim1}
||\Phi^\tau(u(t_n)) - \Phi^{\tau}(u^n) ||_{H^{r}} \le e^{L_n \tau}||e^n||_{H^{r}},
\ee
where
\be\label{L_n}
L_n:= C(||u^n||_{H^{\sigma}}, ||u(t_n)||_{H^{\sigma}}, ||V||_{H^{\sigma}}).
\ee
Using Proposition \ref{LocalS1}, and the estimate \eqref{stab-estim1} a bound of the error term \eqref{err-decomp} is given by,
\be \label{global-estim2}
||e^{n+1}||_{H^{r}} \le M_T\tau^2 + e^{L_n \tau}||e^n||_{H^{r}}, \quad e^0 = 0.
\ee
The global error estimate then easily follows by induction on the above inequality \eqref{global-estim2} once the following uniform bound is obtained:
\be\label{bound_sigma}
\sup_{n\tau \le T} || u^n||_{H^{\sigma}} < +\infty,
\ee
for some $\sigma > \frac{d}{2}$, and for sufficiently small time step $\tau$. In the remainder of the proof we establish the bound \eqref{bound_sigma} for appropriate choices of $\sigma$ depending on the $H^{r}$-norm considered.

In the regime $r > \frac{d}{2}$ we take $\sigma = r$, and the result follows by the classical {\it Lady Windermere's fan} argument (\cite{LW}). Indeed, the uniform bound \eqref{bound_sigma} easily follows for sufficiently small $\tau$ by a bootstrap argument on the estimate \eqref{global-estim2}.
%
%
%

Using a refined global error analysis one can push down the error analysis to the $H^{r}$-norm for $r \le \frac{d}{2}$. 
We take $\sigma = \sigma_0$ where $\sigma_0$ is given in equation \eqref{sigma_0}. In order to show the uniform bound \eqref{bound_sigma} we establish fractional convergence of the scheme \eqref{first-order} in the higher order Sobolev space $H^{\sigma_0}$.
Namely, we show that there exists $\delta > 0$ such that the following estimate holds
\be \label{global-estim3}
||e^{n+1}||_{H^{\sigma_0}} \le M_T\tau^{1+\delta} + e^{L_n \tau}||e^n||_{H^{\sigma_0}}.
\ee
where $L_n= C(||u^n||_{H^{\sigma_0}}, ||u(t_n)||_{H^{\sigma_0}}, ||V||_{H^{\sigma_0}})$. 
Using the decomposition \eqref{err-decomp}, and the bound \eqref{stab-estim1} with $r = \sigma_0$, we are left to show the following local error estimate,
\begin{align}\label{R_sigma_0}
||\mathcal{R}_1(\tau,t_n)||_{H^{\sigma_0}} \le M_T\tau^{1+\delta},
\end{align}
in order to obtain the bound \eqref{global-estim3}.
We obtain the bound \eqref{R_sigma_0} by an interpolation argument. We first show a bound on the remainder $\mathcal{R}^1(\tau,t_n)$ in $H^{r_1}$-norm.
By using Duhamel's formula and by construction of our numerical scheme \eqref{first-order} we have 
$$
\mathcal{R}^1(\tau,t_n) = \int_0^\tau e^{i(\tau-s)\Delta}f(u(t_n)+s, \overline{u}(t_n)+s, V) ds - i \tau e^{i\tau \Delta}\left(u(t_n) \varphi_1(-i \tau \Delta ) V + (u(t_n))^2 \varphi_1(-2i \tau\Delta)\bar{u}(t_n) \right).
$$
One can estimate each of the above terms separately using the first estimate in equation \eqref{A2.2} with $r = r_1$, and $\sigma = \sigma_0$. Indeed, this yields
\begin{align}\label{local-r_1}
||\mathcal{R}^1(\tau,t_n)||_{H^{r_1}} &\le C_{r_1}(||u(t_n)||_{H^{\sigma_0}}, ||u(t_n)||_{H^{r_1}}, ||V|| _{H^{\sigma_0}}) \tau\\ \nonumber
&\le C_{r_1, T} \tau,
\end{align}
where to obtain the last line we recall from equation \eqref{sig_0_b} that $r_1 > \sigma_0 $.
Finally, since $r_1 > \sigma_0 > r$ there exists $\theta \in (0,1)$ such that 
$$
||\mathcal{R}^1(\tau,t_n)||_{H^{\sigma_0}} \le ||\mathcal{R}^1(\tau,t_n)||_{H^{r_1}}^\theta ||\mathcal{R}^1(\tau,t_n)||_{H^{r}}^{1-\theta}.
$$
Using the local error estimates \eqref{local-r_0} and \eqref{local-r_1} we have
\begin{equation*}
||\mathcal{R}^1(\tau,t_n)||_{H^{\sigma_0}} \le M_T \tau^{2-\theta}
\end{equation*}
where $2-\theta>0$. Hence we have shown the bounds \eqref{R_sigma_0} and \eqref {global-estim3} with $\delta = 1-\theta$.
This yields the desired bound \eqref{bound_sigma} with $\sigma = \sigma_0$, by a classical bootstrap argument on equation \eqref{global-estim3}. The first order convergence of the scheme \eqref{first-order} follows by induction using the global bound \eqref{global-estim2}, with $\sigma = \sigma_0$. 
\end{proof}
\section{Second order scheme and analysis}\label{second}
The idea to derive a higher order scheme is to iterate Duhamel's formula \eqref{Duhamel}, and to Taylor expand $f$ around $e^{i\tau \Delta}v$, where we let $v=u_0$. For a second order scheme this yields the following expansion,
%
\begin{align}\label{secondOrder}
u(\tau) 
&=e^{i\tau \Delta}\left[ v -i\mathcal{J}_1(\tau, \Delta, v) -i \int_0^\tau e^{-i\zeta_1 \Delta} [D_{1}f(e^{i\zeta_{1} \Delta}v, e^{-i\zeta_{1} \Delta}\bar{v},V) \cdot e^{i\zeta_{1} \Delta} \mathcal{J}_1(\zeta_{1}, \Delta, v)] d\zeta_{1} \right. \\ \nonumber
& \left.+ i\int_0^\tau e^{-i\zeta_1 \Delta} [D_{2}f(e^{i\zeta_{1} \Delta}v, e^{-i\zeta_{1} \Delta}\bar{v},V) \cdot e^{-i\zeta_{1} \Delta} \overline{\mathcal{J}_1(\zeta_{1}, \Delta, v)}] d\zeta_{1}  \right] + O(\tau^3)
\end{align}
where 
\be\label{derivs}
D_{1} f(v, \bar{v}, V) = -i(V + 2v\bar{v}), \ \text{and}  \ D_{2}f(v,\bar{v}, V) = -iv^2,
\ee
and where we Taylor expanded $f$ around $e^{i\tau \Delta}v$ up to second order in order to obtain a remainder of order three.
Next, the aim is to derive a second order approximation to the integrals appearing in the above expansion \eqref{secondOrder}. 

First, we treat the iterated integrals appearing in the above expression, namely the third and fourth term in equation \eqref{secondOrder}. By a standard Taylor expansion we linearize the exponentials appearing in these iterated integrals. For both $v,V \in H^2$, this yields
$$-ie^{i\zeta_{1} \Delta} \mathcal{J}_1(\zeta_{1}, \Delta, v) = \zeta_1 f(v,\bar{v},V) + O(\zeta_{1}^2 \left(\Delta v+\Delta V)\right).$$
Using the above we make the following second order approximation of the iterated integrals in  equation \eqref{secondOrder};
\begin{align} \label{iterint1} 
-i\int_0^\tau e^{-i\zeta_1 \Delta} [D_{1}f(e^{i\zeta_{1} \Delta}v, e^{-i\zeta_{1} \Delta}\bar{v},V) \cdot e^{i\zeta_{1} \Delta} \mathcal{J}_1(\zeta_{1}, \Delta, v)] d\zeta_{1} = \int_0^\tau \zeta_1 D_{1}f(v,\bar{v}, V) &\cdot f(v,\bar{v}, V)d\zeta_1 \\ \nonumber
&+ O\left(\tau^3 (\Delta v+\Delta V)\right), \\ \nonumber
i\int_0^\tau e^{-i\zeta_1 \Delta} [D_{2}f(e^{i\zeta_{1} \Delta}v, e^{-i\zeta_{1} \Delta}\bar{v},V) \cdot e^{-i\zeta_{1} \Delta} \overline{\mathcal{J}_1(\zeta_{1}, \Delta, v)}] d\zeta_{1} = \int_0^\tau \zeta_1 D_{2}f(v,\bar{v}, V) &\cdot \overline{f(v,\bar{v}, V)}d\zeta_1\\ \nonumber
&+ O\left(\tau^3 (\Delta v+\Delta V)\right).
\end{align}

The above calculations motivate the choice of the expansion for $u$ stated in the following lemma.
\begin{lem}\label{ordre2.0}
Let $v = u^0$. At second order $u$ can be expanded as
$$u(\tau) = u_2(\tau) + R_{2,0}(\tau)$$
with
$$u_2(\tau) = e^{i\tau \Delta}v -i e^{i\tau \Delta} \mathcal{J}_1(\tau, \Delta, v) - \frac{\tau^2}{2} e^{i\tau \Delta} (|v|^4v + 3v |v|^2V - |v|^2v\bar{V} + vV^2),
$$
and
\begin{align*}
R_{2,0}(\tau) 
&= \int_0^\tau e^{i(\tau-\zeta_1) \Delta}[ f(u(\zeta_1), \bar{u}(\zeta_1), V) - f(e^{i\zeta_1}v, e^{-i\zeta_1}\bar{v}, V) ] d\zeta_1 \\
& + e^{i\tau \Delta} \int_0^\tau \zeta_1 (|v|^4v + 3v |v|^2V - |v|^2v\bar{V} + vV^2) d\zeta_1.
\end{align*}
\end{lem}

\begin{proof} The result immediately follows 
by recalling the definition of the principal oscillations \eqref{PO}, and Duhamel's formula \eqref{Duhamel}. Moreover, we note that by simple calculations one has 
\begin{align*}
&D_{1}f(v,\bar{v}, V) \cdot f(v,\bar{v},V) = -(V^2v + 3v^2 \bar{v}V + 2v^3 \bar{v}^2), \quad
D_{2}f(v,\bar{v}, V) \cdot \overline{f(v,\bar{v},V)} = \bar{V}\bar{v}v^2 + v^3\bar{v}^2,
\end{align*}
and hence by equation \eqref{derivs} we have,
\be\label{last_term}
D_{1}f(v,\bar{v}, V) \cdot f(v,\bar{v},V) + D_{2}f(v,\bar{v}, V) \cdot \overline{f(v,\bar{v},V)} = -(|v|^4v + 3v |v|^2V - |v|^2v\bar{V} + vV^2).
\ee
The above calculations together with equations \eqref{secondOrder} and \eqref{iterint1} motivates the inclusion of the last term in the expansion of $u_2(\tau)$.
\end{proof}
It remains to establish a low-regularity second order approximation to the principal oscillatory integral \eqref{PO}. 
We first recall from Section \ref{general1} that in order to derive a low-regularity approximation of $\mathcal{J}_1$ at first order we used the filtered function \eqref{filterfcn} and its first order Taylor expansion \eqref{Taylor1st}. Analogously, to obtain a second order approximation of $\mathcal{J}_1$ we Taylor expand equation \eqref{filterfcn} around $s=0$ up to second order, and include the first two terms of this expansion into our scheme. This yields,
\begin{align}\label{filterOrder2}
\mathcal{N}(\tau, \zeta, \zeta, \Delta, v) 
&= \mathcal{N}(\tau, 0, \zeta, \Delta, v) + \zeta \partial_s \mathcal{N}(\tau,0, \zeta, \Delta, v) + \int_0^\zeta \int_0^s \partial_{s_1}^2 \mathcal{N}(\tau, s_1, \zeta, \Delta, v)ds_{1} ds.
\end{align}
Equivalently, using the filtered function \eqref{filterfcn}, the above expression is written as
\begin{align}\label{YvainOrder2}
e^{-i\zeta \Delta} f(e^{i\zeta \Delta}v,e^{-i\zeta \Delta}\bar{v}, V) 
&=  f(v,e^{-2i\zeta\Delta }  \bar{v}, e^{-i\zeta \Delta }V) +\zeta \mathcal{C}[f,i\Delta](v,e^{-2i\zeta\Delta}\bar{v}, e^{-i\zeta \Delta}V) \\ \nonumber
&+\int_0^\zeta \int_0^se^{-is_1\Delta} \mathcal{C}^2[f,i\Delta](e^{is_1 \Delta}v, e^{is_1\Delta}e^{-2i\zeta\Delta}\bar{v},  e^{is_1\Delta}e^{-i\zeta\Delta}V)ds_1ds
\end{align}
where the local error structure is governed by the second-order commutator type term
\begin{align}\label{com-2}
\mathcal{C}^2[f,i\Delta](u,v,w) := \ &\mathcal{C}[\mathcal{C}[f,i\Delta],i\Delta](u,v,w) \\ \nonumber
= &\sum_{l=1}^{d}(\partial_{l}^2u \partial_{l}^2 w) + (2u+v)(\partial_{l}^2u \partial_{l}^2 v) +  (\partial_{l} v)^2 \partial_{l}^2u + 2(\partial_{l} u)^2 \partial_{l}^2v + \partial_{l} v \partial_{l} u (2\partial_{l}^2u + \partial_{l}^2 v).
\end{align}

In practical computations we need to address the stability issues caused by including into the scheme the second term $\mathcal{C}[f,i\Delta](v,e^{-2i\zeta\Delta}\bar{v}e^{-i\zeta \Delta}V)$ which has the form \eqref{com}, since it involves spatial derivatives. 
Different approaches can be made to treat this issue and guarantee the stability of the scheme and in what follows we offer two different approaches. The first approach is based on \cite{RS} and consists in first introducing a stabilization in the Taylor series expansion \eqref{filterOrder2} based on finite difference approximations. The second approach relies on
directly embedding the commutator term appearing in equation \eqref{YvainOrder2} into the discretization, and then stabilizes the scheme a posteriori by the use of a properly chosen filter function (see also the general approach in the work \cite{ABBS}).
\subsection{A first approach to guarantee stability}\label{sec:stab1}
A first approach consists in stabilizing the second term of equation \eqref{filterOrder2}. This may be done by introducing the following finite difference approximation of $\partial_s \mathcal{N}(\tau,0, \zeta, \Delta, v)$:
\be\label{Stab1}
\partial_s \mathcal{N}(\tau,0, \zeta, \Delta, v) = \frac{1}{\tau}(\mathcal{N}(\tau,\tau, \zeta, \Delta, v) - \mathcal{N}(\tau,0, \zeta, \Delta, v))+O(\tau \partial_{s}^2\mathcal{N}(\tau,\eta, \zeta, \Delta, v))
\ee
for some $\eta \in [0,\tau]$, and where
%
\begin{align}\label{N2}
\partial_s^2 \mathcal{N}(\tau,\eta, \zeta, \Delta, v) = e^{-i\eta\Delta} \mathcal{C}^2[f,i\Delta](e^{i\eta \Delta}v, e^{i\eta\Delta}e^{-2i\zeta\Delta}\bar{v},  e^{i\eta\Delta}e^{-i\zeta\Delta}V). 
\end{align}
The above expansion comes into play in the following lemma, where we obtain a stable second order approximation of the principal oscillation $\mathcal{J}_1$.
\begin{lem}\label{stab2.1}
At second order the principal oscillations can be expanded by
\begin{align*}
\mathcal{J}_1(\tau, \Delta, v) 
= \tau &\left(v \varphi_1(-i \tau \Delta ) V + v^2 \varphi_1(-2i \tau\Delta)\bar{v} \right) \\
&+ \tau e^{-i\tau \Delta}\left((e^{i\tau \Delta}v) \varphi_2(-i \tau \Delta ) (e^{i\tau \Delta}V) + (e^{i\tau \Delta}v)^2 \varphi_2(-2i\tau \Delta) (e^{i\tau \Delta}\bar{v}) \right)\\
&-\tau \left(v \varphi_2(-i \tau \Delta ) V + v^2\varphi_2(-2i\tau \Delta)\bar{v} \right) + iR_{2,1}^1(\tau)
\end{align*}
where $\varphi_2(z) = \frac{1}{z}(e^z - \varphi_1(z))$ and the remainder is given by
\begin{align}\label{R^1_2,1}
R_{2,1}^1(\tau) 
&= \int_0^\tau \int_0^\zeta \int_0^se^{-is_1\Delta} \mathcal{C}^2[f,i\Delta](e^{is_1 \Delta}v, e^{is_1\Delta}e^{-2i\zeta\Delta}\bar{v},  e^{is_1\Delta}e^{-i\zeta\Delta}V)ds_1ds d\zeta\\\nonumber
&+ \int_0^\tau \zeta \int_0^1 \int_0^{\tau s}e^{-is_1\Delta} \mathcal{C}^2[f,i\Delta](e^{is_1 \Delta}v, e^{is_1\Delta}e^{-2i\zeta\Delta}\bar{v},  e^{is_1\Delta}e^{-i\zeta\Delta}V)ds_1ds d\zeta.
\end{align}
\end{lem}
%
\begin{proof} Using the filtered function \eqref{filterfcn}, and plugging the finite difference \eqref{Stab1} into the Taylor expansion \eqref{filterOrder2} we obtain
\begin{align*}
\qquad \mathcal{J}_1(\tau, \Delta, v) 
&=\int_0^\tau \mathcal{N}(\tau, \zeta, \zeta, \Delta, v) d\zeta \\
&= \int_0^\tau  \mathcal{N}(\tau,0, \zeta, \Delta, v)d\zeta + \frac{1}{\tau}\int_0^\tau \zeta (\mathcal{N}(\tau,\tau, \zeta, \Delta, v) - \mathcal{N}(\tau,0, \zeta, \Delta, v))d\zeta \\
&\quad + \int_0^\tau \zeta \int_0^1 \int_0^{\tau s} \partial_{s_1}^2 \mathcal{N}(\tau, s_1, \zeta, \Delta, v)ds_{1} dsd\zeta + \int_0^\tau \int_0^\zeta \int_0^s \partial_{s_1}^2 \mathcal{N}(\tau, s_1, \zeta, \Delta, v)ds_{1} dsd\zeta.
\end{align*}
Using the filtered function \eqref{filterfcn}, equation \eqref{N2} and the definition \eqref{R^1_2,1} of $R^1_{2,1}(\tau)$ it follows from the above that,
\begin{align*}
\mkern-36mu\mkern-36mu\mkern-36mu\mkern-36mu \mathcal{J}_1(\tau, \Delta, v) 
&= \int_0^\tau ([e^{-i\zeta \Delta} V]v + [e^{-2i\zeta \Delta} \bar{v}]v^2)d\zeta
 \\&\quad
 + \frac{1}{\tau}e^{-i\tau \Delta}\int_0^\tau \zeta ([e^{-i\zeta \Delta}e^{i\tau \Delta} V] (e^{i\tau \Delta}v) + [e^{-2i\zeta\Delta}e^{i\tau \Delta} \bar{v}](e^{i\tau \Delta}v)^2)d\zeta \\
&\quad - \frac{1}{\tau}\int_0^\tau \zeta ([e^{-i\zeta\Delta} V] v +[e^{-2i\zeta\Delta} \bar{v}]v^2)d\zeta +iR_{2,1}^1(\tau)\\
&= \tau \left(v \varphi_1(-i \tau \Delta ) V + v^2 \varphi_1(-2i \tau\Delta)\bar{v} \right) 
\\&\quad+ \tau e^{-i\tau \Delta}\left((e^{i\tau \Delta}v) \varphi_2(-i \tau \Delta ) (e^{i\tau \Delta}V) + (e^{i\tau \Delta}v)^2 \varphi_2(-2i\tau \Delta) (e^{i\tau \Delta}\bar{v}) \right)\\
&\quad-\tau \left(v \varphi_2(-i \tau \Delta ) V + v^2\varphi_2(-2i\tau \Delta)\bar{v} \right) + iR_{2,1}^1(\tau)
\end{align*}
which concludes the proof.
\end{proof}
%
By merging the two preceding lemmas \ref{ordre2.0} and \ref{stab2.1}, we obtain the following second order low-regularity scheme for \eqref{evGP}.
%
%
 %
\begin{cor}\label{cor:2scheme_FD} The exact solution $u$ of \eqref{evGP} can be expanded as
\begin{align}\label{2scheme_FD}
u(t_n + \tau) 
&= e^{i\tau \Delta}u(t_n) - i \tau e^{i\tau \Delta}\left(u(t_n) \varphi_1(-i \tau \Delta ) V + u(t_n)^2 \varphi_1(-2i \tau\Delta)\bar{u}(t_n) \right) \\\nonumber
&\quad -i\tau \left((e^{i\tau \Delta}u(t_n)) \varphi_2(-i \tau \Delta ) (e^{i\tau \Delta}V) + (e^{i\tau \Delta}u(t_n))^2 \varphi_2(-2i\tau \Delta) e^{i\tau \Delta}\bar{u}(t_n) \right)\\\nonumber
&\quad +i\tau e^{i\tau \Delta}\left(u(t_n) \varphi_2(-i \tau \Delta ) V + u(t_n)^2\varphi_2(-2i\tau \Delta)\bar{u}(t_n) \right) \\ \nonumber
&\quad-\frac{\tau^2}{2} e^{i\tau \Delta}\left(|u(t_n)|^4u(t_n) + 3u(t_n) |u(t_n)|^2V - |u(t_n)|^2u(t_n)\bar{V} + u(t_n)V^2 \right) \\\nonumber
&\quad+ \mathcal{R}^1_2(\tau, t_n)
\end{align}
%
where the remainder is given by 
\be\label{R^1_2}
\begin{split}
\mathcal{R}^1_2(\tau, t_n)
&= \int_0^\tau e^{i(\tau-\zeta) \Delta}[ f(u(t_n +\zeta), \bar{u}(t_n+ \zeta), V) - f(e^{i\zeta}u(t_n), e^{-i\zeta}\bar{u}(t_n), V) ] d\zeta \\
&\quad + e^{i\tau \Delta} \int_0^\tau \zeta \left(|u(t_n)|^4u(t_n) + 3u(t_n) |u(t_n)|^2V - |u(t_n)|^2u(t_n)\bar{V} + u(t_n)V^2 \right) d\zeta\\
&\quad + \int_0^\tau \int_0^\zeta \int_0^se^{i(\tau -s_1)\Delta} \mathcal{C}^2[f,i\Delta](e^{is_1 \Delta}u(t_n), e^{is_1\Delta}e^{-2i\zeta\Delta}\bar{u}(t_n),  e^{is_1\Delta}e^{-i\zeta\Delta}V)ds_1ds d\zeta\\
&\quad + \int_0^\tau \zeta \int_0^1 \int_0^{\tau s}e^{i(\tau -s_1)\Delta} \mathcal{C}^2[f,i\Delta](e^{is_1 \Delta}u(t_n), e^{is_1\Delta}e^{-2i\zeta\Delta}\bar{u}(t_n),  e^{is_1\Delta}e^{-i\zeta\Delta}V)ds_1ds d\zeta.
\end{split}
\ee
and where the explicit expression for the commutator is given in equation \eqref{com-2}. 
\end{cor}
%
%
\subsection{A second approach to guarantee stability}\label{sec:stab2}
We next present a second approach to the second order approximation of the principal oscillations $\mathcal{J}_1$. In contrast to the preceding section we will first write the approximation in terms of the commutator, then stabilize the scheme by the use of a properly chosen filter function. 

\begin{lem}
To second order the principal oscillations can be expanded by
\be
\begin{split}
\mathcal{J}_1(\tau, \Delta, v) 
&= \tau \left(v \varphi_1(-i \tau \Delta ) V + v^2 \varphi_1(-2i \tau\Delta)\bar{v} \right) + \tau^2 \mathcal{C}[f,i\Delta](v,\varphi_2(-2i\tau\Delta)\bar{v}, \varphi_2(-i\tau \Delta)V) \\
&\quad + iR_{2,1}^2(\tau)
\end{split}
\ee
where the remainder is given by,
\begin{align}\label{R_2,1}
R_{2,1}^2(\tau) 
&= \int_0^\tau \int_0^\zeta \int_0^se^{-is_1\Delta} \mathcal{C}^2[f,i\Delta](e^{is_1 \Delta}v, e^{is_1\Delta}e^{-2i\zeta\Delta}\bar{v},  e^{is_1\Delta}e^{-i\zeta\Delta}V)ds_1ds d\zeta.\\ \nonumber
\end{align}
\end{lem}
\begin{proof} Using the definition of the principal oscillations \eqref{PO} and equation \eqref{YvainOrder2} one has the following expansion,
%
\begin{align*}
\mathcal{J}_1(\tau, \Delta, v) 
&= \int_0^\tau f(v,e^{-2i\zeta\Delta }  \bar{v}, e^{-i\zeta \Delta }V)d\zeta + \int_0^\tau \zeta \mathcal{C}[f,i\Delta](v,e^{-2i\zeta\Delta}\bar{v}, e^{-i\zeta \Delta}V)d\zeta \\
&\quad+ \int_0^\tau \int_0^\zeta \int_0^se^{-is_1\Delta} \mathcal{C}^2[f,i\Delta](e^{is_1 \Delta}v, e^{is_1\Delta}e^{-2i\zeta\Delta}\bar{v},  e^{is_1\Delta}e^{-i\zeta\Delta}V)ds_1ds d\zeta \\
&= \tau \left(v \varphi_1(-i \tau \Delta ) V + v^2 \varphi_1(-2i \tau\Delta)\bar{v} \right) + \tau^2 \mathcal{C}[f,i\Delta](v,\varphi_2(-2i\tau\Delta)\bar{v}, \varphi_2(-i\tau \Delta)V)\\
&\quad+ iR^2_{2,1}(\tau).
\end{align*}
%
%
where the second term could be integrated exactly using the structure of the commutator \eqref{com} and of the nonlinearity \eqref{nonlin}
\end{proof}
The following lemma provides the second order low regularity integrator up to this step.
\begin{lem} The exact solution $u$ of \eqref{evGP} can be expanded as
\begin{align}\label{unstab2}
u(t_n + \tau) 
&= e^{i\tau \Delta}u(t_n) - i \tau e^{i\tau \Delta}\left(u(t_n) \varphi_1(-i \tau \Delta ) V + u(t_n)^2 \varphi_1(-2i \tau\Delta)\bar{u}(t_n) \right) \\\nonumber
&\quad -i\tau^2e^{i\tau \Delta} \mathcal{C}[f,i\Delta](u(t_n),\varphi_2(-2i\tau\Delta)\bar{u}(t_n), \varphi_2(-i\tau \Delta)V)\\\nonumber
&\quad -\frac{\tau^2}{2} e^{i\tau \Delta}(|u(t_n)|^4u(t_n) + 3u(t_n) |u(t_n)|^2V - |u(t_n)|^2u(t_n)\bar{V} + u(t_n)V^2 ) \\ \nonumber
&\quad + R_{2,2}^2(\tau, t_n)
\end{align}
where the remainder is given by
\begin{align*}
R_{2,2}^2(\tau) 
&= \int_0^\tau e^{i(\tau-\zeta) \Delta}[ f(u(t_n +\zeta), \bar{u}(t_n+ \zeta), V) - f(e^{i\zeta}u(t_n), e^{-i\zeta}\bar{u}(t_n), V) ] d\zeta \\
& + e^{i\tau \Delta} \int_0^\tau \zeta (|u(t_n)|^4u(t_n) + 3u(t_n) |u(t_n)|^2V - |u(t_n)|^2u(t_n)\bar{V} + u(t_n)V^2 ) d\zeta\\
&+ \int_0^\tau \int_0^\zeta \int_0^se^{i(\tau -s_1)\Delta} \mathcal{C}^2[f,i\Delta](e^{is_1 \Delta}u(t_n), e^{is_1\Delta}e^{-2i\zeta\Delta}\bar{u}(t_n),  e^{is_1\Delta}e^{-i\zeta\Delta}V)ds_1ds d\zeta.
\end{align*}
\end{lem}
%
%
To stabilize the term appearing in the second line of equation  \eqref{unstab2}, which is of the form $\tau^2 \mathcal{C}[f,i\Delta](v, \bar{v}, V)$, we introduce an appropriate filter operator which we denote by $\Psi$. More precisely, we will construct a filter operator of the form
$$
\Psi = \psi(i\tau|\nabla|),
$$
where $\psi$ is a suitably chosen filter function 
which allows to stabilize the scheme while introducing an error term which only requires $H^2$-regularity on the initial data and potential.
Namely, we require the filter function $\psi$ to introduce the same optimal local error of $O(\tau^{3} \Delta (v+V))$ as is introduced by the low-regularity second order scheme up to this step (see equations  \eqref{iterint1}, \eqref{com-2} and Section \ref{sec:localO2} for the thorough analysis). We refer to \cite{LW} for an introduction to filter functions in the ODE setting.
We now present two sufficient assumptions on the filter operator which once established, guarantees the stability of the low-regularity scheme \eqref{unstab2}.

{\bf Assumption 1}. The filter operator $\Psi = \psi(i\tau |\nabla|)$, satisfies the following bound
\be\label{A1}
||\tau \Psi [\mathcal{C}[f,i\Delta](v, \bar{v}, V)] ||_{r} \le C_{r,d, V} ||v||_r^{m}
\ee
for some $m = m(f) \in \mathbb{N}$ and $r = r(d)\ge 0$.

{\bf Assumption 2}. The filter operator $\Psi = \psi(i\tau |\nabla|)$ satisfies the following expansion
\be\label{A2}
\Psi[\mathcal{C}[f,i\Delta](v, \bar{v}, V)] = \mathcal{C}[f,i\Delta](v, \bar{v}, V) + O(\tau |\nabla|^2 (v+V)).
\ee
The condition in Assumption 1 guarantees the stability of the scheme in the $H^r$-norm, while 
the condition in Assumption 2 preserves the optimal local error structure of $O(\tau^{3} |\nabla|^2 (v+V))$ with the inclusion of the filter function $\psi$. This is an essential ingredient for the local and global error analysis of the scheme.
\begin{rem} \label{bilinEstStab}
The stability estimate \eqref{A1} relies on the algebraic structure of the underlying space. In the following stability analysis we will restrict our attention to sufficiently smooth Sobolev spaces $H^{r}$ with $r> \frac{d}{2} + 1$.
This allows us to exploit the following classical bilinear estimate,
$$
||vw||_{s_0} \le C ||v||_{s_0}||w||_{s_0},
$$
where $s_0 = r - 1$.
An analysis in a lower order Sobolev space would require the use of more refined estimates on the commutator term using the generalized Leibniz rule (see \cite[Chapter 2]{BCD}). This analysis is not detailed here, since we tackle the error analysis of the second order scheme \eqref{2scheme_FD}, based upon the first approach (see Section \ref{sec:stab1}).
\end{rem}
%

A choice of filter operator which is well adapted for the second order scheme \eqref{unstab2} is the following.

\begin{lem}
The filter operator 
\begin{equation}\label{ffp2}
\Psi = \varphi_1(i\tau |\nabla|) := \frac{e^{i \tau |\nabla|}-1}{i \tau |\nabla|}
\end{equation}
satisfies Assumption 1 and 2 with $r>\frac{d}{2}+1$.
\end{lem}
\begin{proof}
We first show how the filter function \eqref{ffp2} satisfies Assumption 1 and hence guarantees the stability of the second order low-regularity scheme.
By definition of the $\varphi_1$ function and using the explicit form of the commutator \eqref{com} together with the bilinear estimate we have, 
\begin{align*}
||\tau \varphi_1(i\tau |\nabla|) \mathcal{C}[f,i\Delta](v, \bar{v}, V) ||_{r} 
&\le || (e^{i\tau |\nabla|} -1) |\nabla|^{-1} \mathcal{C}[f,i\Delta](v,\bar{v}, V) ||_{r}\\
&\le 4||\nabla V \cdot \nabla v + |\nabla v|^2 \bar{v} + 2v\nabla v \cdot \nabla \bar{v}||_{r-1}\\
&\le C_{r,d}(||\nabla V||_{r-1} ||\nabla v||_{r-1} + ||\nabla v ||_{r-1}^2 ||v||_{r-1}
\\& \quad+ 2||v||_{r-1} ||\nabla v||_{r-1}^2)\\
&\le C_{r,d}( ||V||_r ||v ||_r + 3 ||v||_{r}^3)\\
&\le C_{r,d, V} (||v||_r + ||v||_r^3).
\end{align*}
%
Furthermore, by a simple Taylor's expansion we have that, 
\begin{equation}\label{A2varphi_1}
\varphi_1(i\tau |\nabla|) = 1 + O(\tau |\nabla|).
\end{equation}
It then follows by the form of the commutator \eqref{com} that the filter function \eqref{ffp2} satisfies Assumption 2. Indeed, from the above equation we have,
$$
\varphi_1(i\tau \nabla) \mathcal{C}[f,i\Delta](v, \bar{v}, V) = \mathcal{C}[f,i\Delta](v, \bar{v}, V) + O(\tau |\nabla| \mathcal{C}[f,i\Delta](v, \bar{v}, V)),
$$
where by equation \eqref{com} we have that formally $O(\tau |\nabla| \mathcal{C}[f,i\Delta](v, \bar{v}, V)) =  O(\tau |\nabla|^2(v + V))$.
Hence, formally we see that the inclusion of the filter function \eqref{ffp2} preserves the optimal local error, by only requiring two additional derivatives on the initial datum and the potential (see Proposition \ref{LocalS2}, in the regime $r> \frac{d}{2}$).
 \end{proof}
The following Corollary provides a {\it stable} second order low-regularity scheme for \eqref{evGP}, by using the filter function \eqref{ffp2}.
\begin{cor}\label{cor:stab2} The exact solution $u$ of \eqref{evGP} can be expanded as
\be\label{stab2}
\begin{split}
u(t_n + \tau) 
&= e^{i\tau \Delta}u(t_n) - i \tau e^{i\tau \Delta}\left(u(t_n) \varphi_1(-i \tau \Delta ) V + u(t_n)^2 \varphi_1(-2i \tau\Delta)\bar{u}(t_n) \right) \\
&-i\tau^2e^{i\tau \Delta} \varphi_1(i\tau \nabla)[ \mathcal{C}[f,i\Delta](u(t_n),\varphi_2(-2i\tau\Delta)\bar{u}(t_n), \varphi_2(-i\tau \Delta)V)]\\
&-\frac{\tau^2}{2} e^{i\tau \Delta}(|u(t_n)|^4u(t_n) + 3u(t_n) |u(t_n)|^2V - |u(t_n)|^2u(t_n)\bar{V} + u(t_n)V^2 ) \\
&+ \mathcal{R}^2_2(\tau, t_n).
\end{split}
\ee
where
\begin{align*}
\mathcal{R}^2_2(\tau, t_n) 
&= \int_0^\tau e^{i(\tau-\zeta) \Delta}[ f(u(t_n +\zeta), \bar{u}(t_n+ \zeta), V) - f(e^{i\zeta}u(t_n), e^{-i\zeta}\bar{u}(t_n), V) ] d\zeta \\
& + e^{i\tau \Delta} \int_0^\tau \zeta (|u(t_n)|^4u(t_n) + 3u(t_n) |u(t_n)|^2V - |u(t_n)|^2u(t_n)\bar{V} + u(t_n)V^2 ) d\zeta\\
&+ \int_0^\tau \int_0^\zeta \int_0^se^{i(\tau -s_1)\Delta} \mathcal{C}^2[f,i\Delta](e^{is_1 \Delta}u(t_n), e^{is_1\Delta}e^{-2i\zeta\Delta}\bar{u}(t_n),  e^{is_1\Delta}e^{-i\zeta\Delta}V)ds_1ds d\zeta\\
&-i\tau^2e^{i\tau\Delta}(I - \varphi_1(i\tau \nabla))[ \mathcal{C}[f,i\Delta](u(t_n),\varphi_2(-2i\tau\Delta)\bar{u}(t_n), \varphi_2(-i\tau \Delta)V)].
\end{align*}
\end{cor}
\subsection{Local error estimates}\label{sec:localO2}
In this section we prove that the second order scheme \eqref{2scheme_FD} introduces a local error of order three under favorable regularity assumptions of the initial datum and potential. 
As was the case for the error analysis of the first-order scheme (Section \ref{sec:LocalS1}), we make the analysis in $H^{r}$-norm where the regularity assumptions on $v$ and $V$ will depend on the regime of $r$ considered.
%
\begin{prop}\label{LocalS2}
Let $T> 0$, $r \ge 0$, and $r_2$ as in Theorem \ref{Global2}, namely
\begin{equation*}
r_2 = 
\begin{cases}
r + 2, \ \ \text{if} \ r > \frac{d}{2},\\
2+ \frac{d}{2} + \epsilon, \ \text{if} \ 0 < r \le \frac{d}{2},\\
2 + \frac{d}{4},  \ \ \text{if} \ r = 0,\\ 
\end{cases}
\end{equation*}
where $0 < \epsilon < \frac{1}{4}$ can be arbitrarily small.
Assume there exists $C_T>0$ such that
\begin{equation*}
\sup_{[0,T]}||u(t)||_{H^{r_2}} \le C_T, \ \text{and} \ \ 
||V||_{H^{r_2}} \le C_T,
\end{equation*}
then there exists $M_T >0$ 
such that for every $\tau \in (0,1]$
$$||\mathcal{R}^1_2(\tau, t_n)||_{H^{r}} \le M_T \tau^3, \quad 0 \le t_n \le T,
$$
where $t_n = n\tau$, and $\mathcal{R}^1_2(\tau, t_n)$ is given in equation  \eqref{R^1_2}.
\end{prop}
\begin{proof} 
We write the error term $\mathcal{R}^1_2(\tau, t_n)$, as the sum of four terms, $\mathcal{R}^1_2(\tau, t_n) = \mathcal{E}^1(\tau,t_n) + \mathcal{E}^2(\tau,t_n) +  \mathcal{E}^3(\tau,t_n) + \mathcal{E}^4(\tau,t_n)$. We start by establishing the third order estimate for the two last terms $\mathcal{E}^3$ and $\mathcal{E}^4$. These bounds are obtained using the same arguments as those made to bound the term $\mathcal{G}^2$ in the proof of Proposition \ref{LocalS1}, by noticing that $r_2 = r_1+1$. Indeed, using the inequality \eqref{bilin-nonsmooth}, and the explicit expression of the second-order commutator \eqref{com-2} we obtain that for $r \ge 0$ and $\sigma > \frac{d}{2}$,
\be
||\mathcal{C}^2[f,i\Delta](u,v,w)||_{H^{r}} \le C(||u||_{H^{\sigma+2}}, ||v||_{H^{\sigma+2}}, ||u||_{H^{r+2}}, ||v||_{H^{r+2}},||w||_{H^{r+2}}).
\ee
Following the proof of Proposition \ref{LocalS1}, in the regime $r > \frac{d}{2}$ one can conclude by taking $\sigma = r$, and in the regime $0<r \le \frac{d}{2}$ by taking $\sigma = \sigma_0$, where $\sigma_0$ is defined in equation  \eqref{sigma_0}, and by recalling from  equation  \eqref{sig_0_b} that $r + 2 < \sigma_0 + 2 < r_2 = \frac{d}{2} + \epsilon + 2$. Finally the case $r = 0$ again follows by using the embedding $H^{\frac{d}{4}} \hookrightarrow L^4$.
Next, we show that the sum of the remaining terms, $(\mathcal{E}^1+ \mathcal{E}^2)(\tau, t_n)$, of equation  \eqref{R^1_2} is of third order.
We have that $r_2 > 2$, and hence $u(t), V \in H^2$. Hence, by Taylor expanding the exponential appearing inside the Duhamel's integral \eqref{Duhamel} we obtain the following expansion :
$u(t_n + \zeta) = e^{i\zeta \Delta}u(t_n) + \zeta f^n + \tilde{R}(\zeta, t_n)$
where $f^n = f(u(t_n),\bar{u}(t_n), V)$ and 
\be\label{R_2,a}
\tilde{R}(\zeta, t_n) = \int_0^\zeta e^{i(\zeta-s)\Delta} f(u(t_n+s), \bar{u}(t_n+s),V)ds - \zeta f^n.
\ee
%
%
Using the above expansion for $u$ we rewrite the error term $(\mathcal{E}^1+ \mathcal{E}^2)(\tau,t_n)$ as,
\begin{align}\label{E^1}
(\mathcal{E}^1 + \mathcal{E}^2)(\tau, t_n) = \int_0^\tau e^{i(\tau-\zeta) \Delta}[ f\left(e^{i\zeta \Delta}u(t_n) + \zeta f^n + \tilde{R}(\zeta, t_n), \right.&\left. e^{-i\zeta \Delta}\bar{u}(t_n) + \zeta \overline{f^n} + \overline{\tilde{R}(\zeta, t_n)}, V\right) \\ \nonumber
&  - f(e^{i\zeta\Delta}u(t_n), e^{-i\zeta\Delta}\bar{u}(t_n), V) ] d\zeta \\ \nonumber
& - e^{i\tau \Delta} \int_0^\tau \zeta (D_{1}f^n \cdot f^n + D_{2}f^n\cdot \overline{f^n}) d\zeta,
\end{align}
where $D_{1}f^n\cdot f^n + D_{2}f^n\cdot \overline{f^n}$ is given by equation  \eqref{last_term}.
For notational convenience we let $a_1 := e^{i\zeta \Delta}u(t_n)+ \zeta f^n$.
The idea in order to show that the above error term \eqref{E^1} is of third-order revolves around making three suitable Taylor expansions. 
%
%
By Taylor expanding $f$ around $(a_1, \bar{a}_1, V)$ and $(e^{i\zeta\Delta} u(t_n),e^{-i\zeta\Delta} \bar{u}(t_n),V)$ respectively we obtain,
\begin{align}\label{Taylor12}
f\left(a_1 + \tilde{R}(\zeta, t_n), \bar{a}_1 + \overline{\tilde{R}(\zeta, t_n)}, V\right) = f\left(a_1, \bar{a}_1, V\right) + E_{1}(\zeta), \\ \nonumber
f\left(a_1, \bar{a}_1, V\right) = f\left(e^{i\zeta\Delta} u(t_n), e^{-i\zeta\Delta} \bar{u}(t_n), V\right) + E_{2}(\zeta),
\end{align}
where
\begin{align}\label{E1}
&E_1(\zeta) = \int_0^1 D_1f\left(a_1 + \theta \tilde{R}(\zeta,t_n), \bar{a}_1 + \theta\overline{\tilde{R}(\zeta,t_n)}, V \right)\cdot \tilde{R}(\zeta,t_n) \\ \nonumber
&\qquad\qquad\qquad+ D_2 f\left( a_1 + \theta \tilde{R}(\zeta,t_n), \bar{a}_1 + \theta\overline{\tilde{R}(\zeta,t_n)}, V \right)\cdot \overline{\tilde{R}(\zeta,t_n)}d\theta, \\\label{E2}
&E_2(\zeta) = \zeta\int_0^1 [D_1f\left(e^{i\zeta\Delta} u(t_n)+ \theta\zeta f^n,e^{-i\zeta\Delta} \bar{u}(t_n)+\theta\zeta\overline{f^n}, V\right)\cdot f^n \\ \nonumber
& \qquad\qquad\qquad + D_2f\left(e^{i\zeta\Delta} u(t_n)+ \theta\zeta f^n,e^{-i\zeta\Delta} \bar{u}(t_n)+\theta\zeta\overline{f^n}, V\right)\cdot \overline{f^n} ]d\theta.
\end{align}
Hence, plugging equation  \eqref{Taylor12} into equation \eqref{E^1} yields,
\begin{align}\label{ErrTaylored}
(\mathcal{E}^1+\mathcal{E}^2)(\tau, t_n) =&
 \int_0^\tau e^{i(\tau - \zeta)\Delta} E_{1}(\zeta) d\zeta
+ e^{i\tau \Delta}\int_0^\tau e^{-i\zeta \Delta}E_2(\zeta) d\zeta -  e^{i\tau \Delta}\int_0^\tau  \zeta(D_{1}f^n \cdot f^n + D_{2}f^n\cdot \overline{f^n})d\zeta.
\end{align}
In order to show that the first term in the above equation is of third-order we show the bound $||E_{1}(\zeta) ||_{r} \le C_T\zeta^2$.
By equation \eqref{E1} and by using the bilinear inequality \eqref{bilin-nonsmooth} we have that for all $\sigma >\frac{d}{2}$ and $r \ge 0$,
%
\be\label{E1.1_norm}
\begin{split}
||E_{1}(\zeta)||_{r} 
&\le \sup_{\theta \in ]0,1[} \biggl( ||D_1f \left(a_1 + \theta \tilde{R}(\zeta,t_n), \bar{a}_1 + \theta\overline{\tilde{R}(\zeta,t_n)}, V \right)||_{\sigma}\\
&\qquad\qquad+ ||D_2f \left(a_1 + \theta \tilde{R}(\zeta,t_n), \bar{a}_1 + \theta\overline{\tilde{R}(\zeta,t_n)}, V \right)||_{\sigma} \biggr) 
||\tilde{R}(\zeta, t_n) ||_{r} \\
&\le C_{r}\biggl(||V||_{\sigma} , \ \sup_{t \in [0,T]} ||u(t)||_{\sigma}, \sup_{(\zeta, t)\in [0,\tau]\times [0,T] } ||\tilde{R}(\zeta, t)||_{\sigma}\biggr)
||\tilde{R}(\zeta,t_n) ||_{r}
\end{split}
\ee
where the last inequality follows by using the explicit form of the derivatives \eqref{derivs}, the first estimation of equation \eqref{A2.2}, and the fact that $e^{i\zeta\Delta}$ is an isometry on Sobolev spaces.
Next, we show that,
\be\label{E2}
\sup_{(\zeta,t)\in  [0,\tau]\times[0,T]} ||\tilde{R}(\zeta,t) ||_{\sigma_0} < +\infty, \quad \text{and} \quad
||\tilde{R}(\zeta,t_n) ||_{r} \le C_T \zeta^2,
\ee
where $\sigma_0$ is given by equation \eqref{sigma_0}.
%
%
%
We obtain the first bound by using the first estimate of equation \eqref{A2.2} on $f$ with $r = \sigma_0$,
$$
\sup_{(\zeta, t)\in  [0,\tau]\times[0,T]}||\tilde{R}(\zeta,t) ||_{\sigma_0} \le \tau C_{\sigma_0}(\sup_{t \in [0,T]}||u(t)||_{\sigma_0}, ||V||_{\sigma_0}) < +\infty.
$$
Next, we obtain the second estimate in equation \eqref{E2} using the Fundamental Theorem of Calculus. By letting $u = u(t_n+s_1)$ we have,
%
\begin{align}\label{FTC}
\tilde{R}(\zeta, t_n) 
&= \int_0^\zeta \int_0^s \partial_{s_1}\left(e^{i(\zeta-s_1)\Delta} f(u, \bar{u},V) \right) ds_1ds + \zeta e^{i\zeta\Delta}f^n - \zeta f^n \\ \nonumber
& = \int_0^\zeta \int_0^s e^{i(\zeta-s_1)\Delta} \biggl( i\Delta f(u,\bar{u},V) + D_1f(u,\bar{u},V) \left(i\Delta u + f(u,\bar{u},V)\right) \biggr. \\ \nonumber
& \biggl.\quad+ D_2 f(u,\bar{u},V) (-i\overline{\Delta u} +\overline{f(u,\bar{u},V)}) \biggr) ds_1 ds + i\zeta^2\varphi_1(i\zeta\Delta)\Delta f(u,\bar{u},V),
\end{align}
where to obtain the second line we used equation \eqref{evGP}, and for the last term we used the definition of the $\varphi_1$ operator.
The second estimate of equation \eqref{E2} then follows immediately from equation \eqref{FTC} by observing that,
$$
||\Delta f(v,\bar{v},V) ||_{r} + ||D_i f(v,\bar{v},V) \Delta u||_{r} \le C_{r}(||v ||_{r_2}, ||V ||_{r_2}), \quad i= 1,2.
$$

It remains to show that the difference of the second and third term in equation \eqref{ErrTaylored} is of third-order. Formally, this directly follows by making a Taylor expansion of $e^{i\zeta\Delta}u(t_n) + \theta\zeta f^n$ around $\zeta=0$: $e^{i\zeta\Delta}u(t_n) + \theta\zeta f^n = u(t_n) + O(\zeta\Delta u(t_n))$. By using the same application of the Fundamental Theorem of Calculus as done in equation \eqref{FTC} one obtains this third order bound. This concludes the proof.
%
%
\end{proof}
%
%
%
%
%
\subsection{Global error estimate}\label{Gerr2_section}
Using the local error estimates established in the preceding section together with a stability argument we show global second order convergence of our low regularity integrator under the regularity assumptions established in Proposition \ref{LocalS2}.
\begin{proof}[Proof of Theorem \ref{Global2}]
We let $u^{n+1} = \Phi_{\text{num},2}^{\tau}(u^n)$ be the numerical scheme defined in equation \eqref{second-order}.
The outline of the proof of this second-order convergence result follows exactly the same lines as the first order convergence result given in Theorem \ref{Global1}. Indeed, in the case where $r > \frac{d}{2}$ the global error estimate follows by a classical {\it Lady Windermere's argument} (\cite{LW}). In the regime $r \le \frac{d}{2}$, by exploiting the same interpolation argument as made in the proof of Theorem \ref{Global1}, and by using definition \eqref{R^1_2} of $\mathcal{R}^1_2$ we have that there exists an $\delta > 0$ such that
$$
||\mathcal{R}_2^1(\tau,t_n)||_{H^{\sigma_0}} \le M_T\tau^{1+\delta},
$$
where $\sigma_0$ is defined in equation \eqref{sigma_0}.
From the above we obtain the bound \eqref{global-estim3} and hence the apriori $H^{\sigma_0}$-bound on the iterates: $\sup_{n\tau\le T} ||u^n||_{H^{\sigma_0}} < +\infty$, for $\tau$ sufficiently small. Using the local error analysis given in Proposition \ref{LocalS2} we obtain second order convergence of the scheme  \eqref{second-order} by performing an inductive argument with
\be
||e^{n+1}||_{H^{r}} \le M_T\tau^3 + e^{L_n \tau}||e^n||_{H^{r}}, \quad e^0=0,
\ee
since $\displaystyle \sup_{n\tau \le T} L_n\le C_{T, r}(||u^n||_{H^{\sigma_0}}) < + \infty$. This concludes the proof.
%
%
%
\end{proof}
 \section{ Numerical Experiments}\label{sec:numexp}
In this section we provide some numerical experiments to support our theoretical convergence results. We consider the Gross-Pitaevskii equation \eqref{evGP} with an initial data of the form
\be\label{ic:num}
u_0(x) = \sum_{k\in\mathbb{Z}}(1+|k|)^{-\vartheta - \frac{1}{2}}a_ke^{ikx} ,
\ee
where the coefficients $(a_k)_{k\in \mathbb{Z}}$ are chosen as uniformly distributed random complex numbers in $[0,1] + i [0,1]$, using the mathlab function rand. The parameter $\vartheta \ge 0$ dictates the regularity assumption on the above function \eqref{ic:num}, namely it insures that $u_0 \in H^{\vartheta}$.
We choose the potential $V$ to have the same form \eqref{ic:num}, as the initial condition.

For the space discretization, we couple the first and second order low-regularity time integrators \eqref{first-order} and \eqref{second-order} with a standard Fourier pseudo-spectral method. We take the largest Fourier mode as $K = 2^{10}$, yielding a spatial mesh size of $\Delta x = 0.0061$.

In order to test our convergence result in each of the three regimes $r= 0 $, $0<r\le \frac{d}{2}$, and $r>\frac{d}{2}$, we choose to measure the error in the (discrete) $L^2$, $H^{\frac{1}{2}}$, and $H^1$ norms. For each of these three norms we plot the first and second order low regularity integrators given in equations \eqref{first-order} and \eqref{second-order} for $u_0, V \in H^{r_1}$ and $u_0, V\in H^{r_2}$ respectively (see equations \eqref{r_1} and \eqref{r_2}). The error at time $T=1$ are given in Figure \eqref{fig:num}.
The results of our numerical experiments agree with the corresponding theoretical convergence results: we observed first and second order convergence for the regularity assumptions given in Theorem \ref{Global1} and \ref{Global2}. Moreover, as expected, the yellow lines in Figure \ref{fig:num} show how the second order scheme exhibits order reduction for the less regular data and potential $u_0,V \in H^{r_1}$. Nevertheless, it successfully converges to second order for $u_0,V\in H^{r_2}$. We lastly note that the observed convergence is slightly better than predicted by Theorem \ref{Global1} and \ref{Global2} (see as well the work of \cite{OS1}).
%
%
\begin{figure}
\begin{minipage}{.5\linewidth}
\centering
\subfloat[]{\label{main:a}\includegraphics[scale=.43]{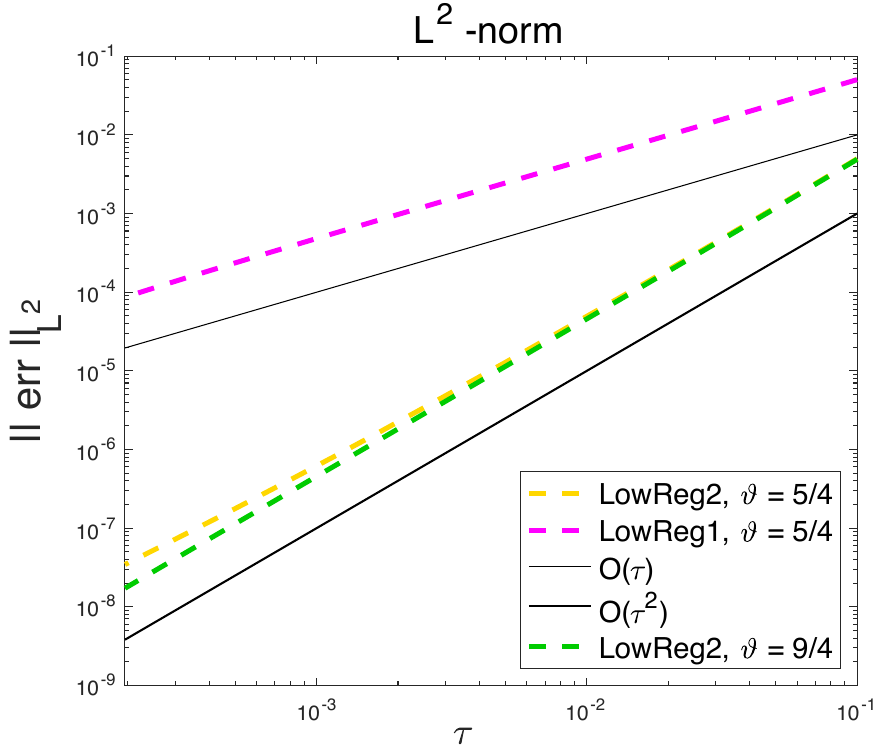}}
\end{minipage}%
\begin{minipage}{.5\linewidth}
\centering
\subfloat[]{\label{main:b}\includegraphics[scale=.43]{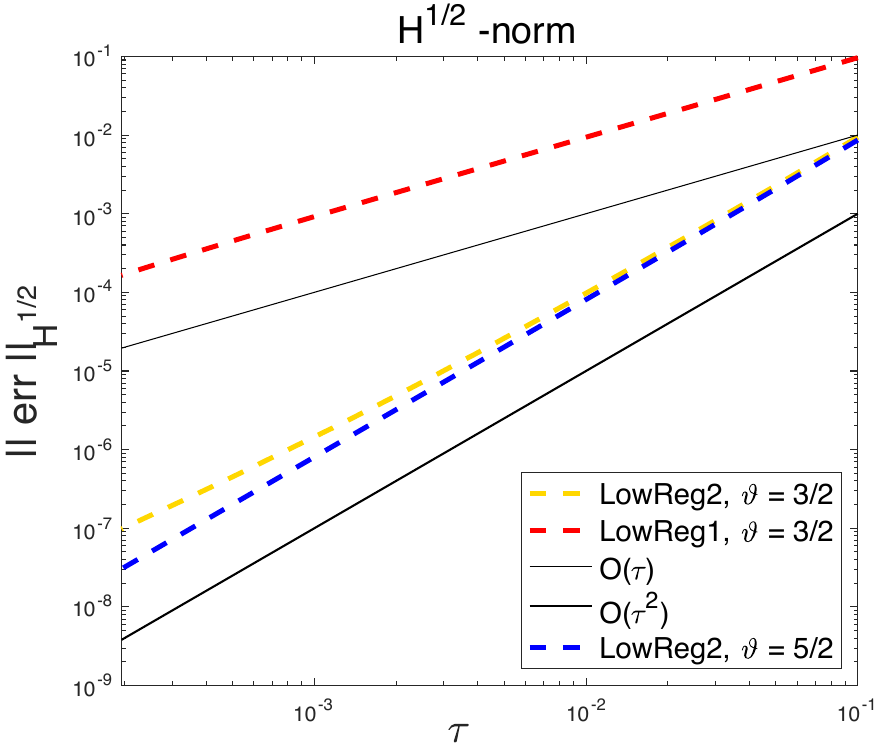}}
\end{minipage}\par\medskip
\centering
\subfloat[]{\label{main:c}\includegraphics[scale=.43]{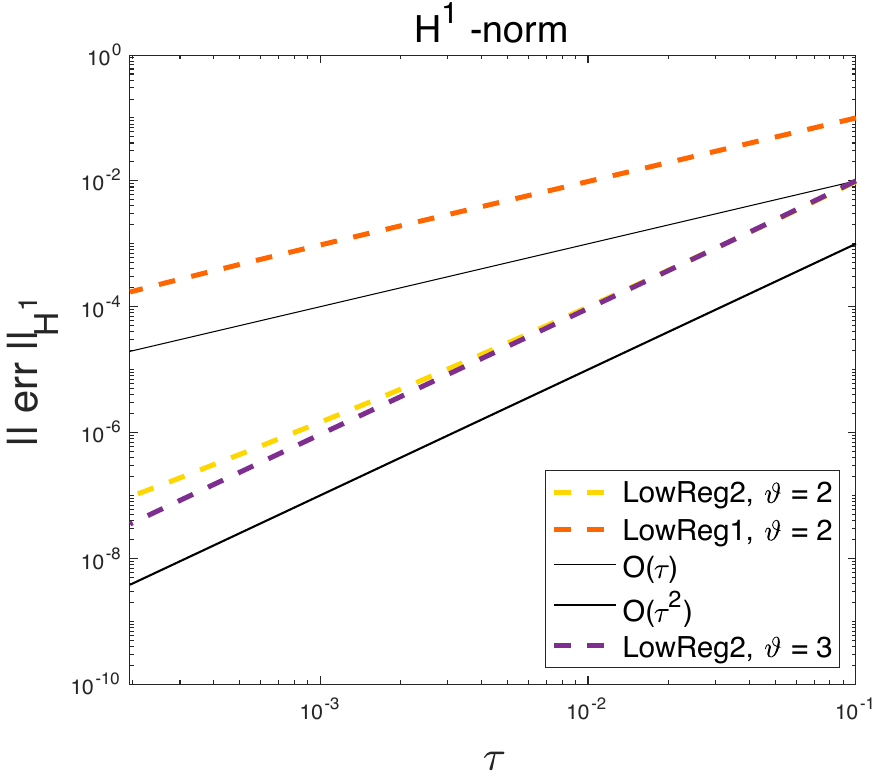}}

\caption{
Convergence plots for three different norms. The slopes of the continuous lines are one and two, respectively. Plot (a) : first and second order scheme with $u_0, V \in H^{5/4}$ (pink and yellow), and for the second-order scheme with $u_0, V \in H^{9/4}$ (green). Plot (b): first and second order scheme with $u_0, V \in H^{\frac{3}{2}}$ (red and yellow), and for the second order scheme with $u_0, V \in H^{5/2}$ (blue). Plot (c): first and resp. second order schemes with $u_0, V \in H^2$ (orange and yellow) and second order with $u_0,V \in H^3$ (purple).
}
\label{fig:num}
\end{figure}

\subsection*{Acknowledgements}
{\small
This project has received funding from the European Research Council (ERC) under the European Union's Horizon 2020 research and innovation programme (grant agreement No. 850941).
The author would like to express her thanks to Katharina Schratz for helpful discussions, and to the anonymous referees for their insightful comments and suggestions.
}  
\section*{Appendix}
In this section we derive the bilinear $H^r$-estimates, for $r> 0$, as stated in Section \ref{notation}. Namely we show the estimate \eqref{bilin-nonsmooth} in both regimes $r> \frac{d}{2}$, and $0<r\le \frac{d}{2}$. 
The notation and tools we use are based upon Littlewood-Paley Theory (see \cite[Chapter 2]{BCD}). We will apply this machinery to our study on the torus $\mathbb{T}^d$. 
Given any tempered distribution $u$, the Littlewood-Paley theory provides the following decomposition,
$$
u = \sum_{k \ge -1} \Delta_k u, \quad \mathcal{F}(\Delta_ku)(\xi) = \varphi_k(\xi)\hat{u}(\xi),
$$
where $\varphi_k(\xi) = \varphi(\xi/ 2^k),$ for $k \ge 0$, $\varphi_{-1} = \chi $, and $\varphi, \chi$ satisfy the assertions in \cite[Proposition 2.10]{BCD}, namely they form a dyadic partition of unity.
In the above and in the remainder of this section we use the fact that, as is the case when working on $\mathbb{R}^d$, one can make use of the Fourier transform on $\mathbb{T}^d$, where in the periodic case we have that $\xi \in \mathbb{Z}^d$.

Using the above Littlewood-Paley decomposition, we introduce Bony's decomposition in order to express the product of two tempered distributions $uv$ as the following sum of three terms,
\begin{equation}\label{bony}
uv = T_u(v) + T_v(u) + R(u,v)
\end{equation}
where $\displaystyle  T_u(v) = \sum_{j} S_{j-1}u\Delta_{j}v$, $\displaystyle S_{j-1}u = \sum_{i\le j-2}\Delta_iu$, and the remainder $\displaystyle R(u,v) = \sum_{|k-j|\le 1}\Delta_ku \Delta_j v$.

In what follows we shall make use of the embedding $H^r \hookrightarrow B^{r-\frac{d}{2}}_{\infty, \infty}$, for non homogeneous Besov spaces, (see \cite[Proposition 2.71]{BCD}).

We are now ready to demonstrate the estimate \eqref{bilin-nonsmooth} in the regime $0<r\le\frac{d}{2}$. Let $\epsilon > 0$. First by \cite[Proposition 2.85]{BCD} together with the embedding $H^{\frac{d}{2}} \hookrightarrow B^{0}_{\infty, \infty}$, 
we have the following estimate of the remainder,
$$
||R(u,v)||_{H^r} \le C_{r,d}||v||_{H^r}||u||_{H^{\frac{d}{2}}}, \quad r > 0.
$$
%
Next, by using the first estimate of \cite[Theorem 2.82]{BCD} together with the embedding $H^{\frac{d}{2} + \epsilon} \hookrightarrow L^{\infty}$
we have the following estimate on the paraproduct of $v$ by $u$,
$$
||T_{u}v||_{H^r} \le ||u||_{H^{\frac{d}{2}+\epsilon}}||v||_{H^r}, \quad r>0.
$$
Using the second estimate of \cite[Theorem 2.82]{BCD} we obtain that for any $0<r\le \frac{d}{2}$ the following estimate of the paraproduct of $u$ by $v$ holds,
\begin{equation*}
||T_{v}u||_{H^r} \le
\begin{cases}
||u||_{H^{\frac{d}{2}}}||v||_{H^r}, \quad \text{for} \ 0<r<\frac{d}{2},\\
||u||_{H^{\frac{d}{2}+\epsilon}}||v||_{H^{\frac{d}{2}}}, \quad \text{if }\ r=\frac{d}{2},
\end{cases}
\end{equation*}
where we used the embedding $H^r \hookrightarrow B^{r-d/2}_{\infty,\infty}$, 
and respectively
$H^{\frac{d}{2} - \epsilon} \hookrightarrow B^{-\epsilon}_{\infty,\infty}$ in the case $r=\frac{d}{2}$. 
%
Therefore, when $0<r\le \frac{d}{2}$, by collecting the above bounds, together with the decomposition \eqref{bony}, we recover the estimate \eqref{bilin-nonsmooth}:
$$
||uv||_{H^r} \le C_{r,d} ||u||_{\frac{d}{2}+\epsilon}||v||_{H^r}.
$$

Finally, we show the estimate \eqref{bilin-nonsmooth} in the regime $r> \frac{d}{2}$, with $\sigma = r$.
By \cite[Corollary 2.86]{BCD} we have,
$$
||uv||_{H^r} \le \frac{C^{r+1}}{r}(||u||_{L^\infty}||v||_{H^r} + ||u||_{H^r}||v||_{L^\infty}), \quad r > 0.
$$
Hence, given any $r> \frac{d}{2}$, using the embedding $H^{r} \hookrightarrow L^\infty$, we obtain the claimed estimate \eqref{bilin-nonsmooth}, with $\sigma = r$.


\begin{thebibliography}{99}

\bibitem{OberwolfachReport}
Y. ~Alama Bronsard, Y.~Bruned, K.~Schratz,
{\it Low regularity integrators for the
Gross-Pitaevskii equation}, Mathematisches Forschungsinstitut Oberwolfach, Report No. 17/2021 DOI: 10.4171/OWR/2021/17.

\bibitem{ABBS}
Y. ~Alama Bronsard, Y.~Bruned, K.~Schratz, {\it Low regularity integrators via decorated trees},  preprint (2022), arxiv.org/abs/2202.01171.

\bibitem{BCD}
H. Bahouri, J.-Y. Chemin, R. Danchin. {\it Fourier Analysis and Nonlinear Partial Differential Equations}. Springer, Heidelberg, (2011).

\bibitem{BS}
Y.~Bruned, K.~Schratz, \textit{Resonance based schemes for dispersive equations via decorated trees}, to appear in Forum of Mathematics, Pi (2022).

\bibitem{CS_KG}
M. Cabrera Calvo, K. Schratz. {\it Uniformly accurate low regularity integrators for the Klein-Gordon equation from the classical to non-relativistic limit regime}, preprint (2021), arxiv:2104.11672.

\bibitem{CM}
R. Coifman and Y. Meyer, {\it On commutators of singular integrals and bilinear singular integrals}, Trans. Amer. Math. Soc.
212 (1975), 315--331. doi:10.1090/S0002-9947-1975-0380244-8. 
%

\bibitem{KIOPS}
S. Gaudreault, G. Rainwater and M. Tokman, {\it KIOPS : A fast adaptive Krylov subspace solver for exponential integrators}, J. Comput. Phys., 372 (2018), 236--255.

\bibitem{LW}
E.~Hairer, C.~Lubich, G.~Wanner, {\it Geometric Numerical Integration. Structure-Preserving Algorithms for Ordinary Differential Equations.} Second Edition. Springer Berlin (2006).


\bibitem{Peterseim}
P. Henning and D. Peterseim. {\it Crank-Nicolson Galerkin approximations to nonlinear Schr\"odinger equations with rough potentials.} M3AS Math. Models Methods Appl. Sci. 27(11):2147-2184, (2017).

\bibitem{HerrSchratz}
S. Herr, K. Schratz, {\it Trigonometric time integrators for the Zakharov system.} Ima. J. Numer. Anal. 37:2042-2066 (2017).
%


\bibitem{ExpInt}
M. Hochbruck, A. Ostermann, {\it Exponential integrators}, Acta Numerica, 19, 209-286. doi:10.1017/S0962492910000048 (2010).


\bibitem{HofS}
M. Hofmanov\'a, K. Schratz, {\it An oscillatory integrator for the KdV equation}, Numer. Math. 136:1117-1137 (2017).

\bibitem{LubichGP}
T.~Jahnke, C~ Lubich,
\textit{Error bounds for exponential operator splittings}. BIT 40:735--744  (2000).
%



\bibitem{Disorder1}
B. Nikolic, A. Balaz, and A. Pelster. {\it Dipolar Bose-Einstein condensates in weak anisotropic disorder.} Physical Review A - Atomic, Molecular, and Optical Physics, 88(1), 2013.

%
\bibitem{OS1}
A.~Ostermann, K.~Schratz, 
\textit{Low regularity exponential-type integrators for semilinear Schr\"odinger equations}, Found Comput Math 18, 731--755 (2018). https://doi.org/10.1007/s10208-017-9352-1

\bibitem{ORS1}
A.~Ostermann, F.~Rousset, K.~Schratz. 
\textit{Error estimates of a Fourier integrator for the cubic Schr\"odinger equation at low regularity}, Found Comput Math 21, 725--765 (2021). https://doi.org/10.1007/s10208-020-09468-7

\bibitem{ORS2}
A. Ostermann, F. Rousset, K. Schratz. {\it Fourier integrator for periodic NLS: low regularity estimates via discrete Bourgain spaces.}
http://arxiv.org/abs/2006.12785 to appear in J. Eur. Math. Soc. (JEMS)

\bibitem{OWY}
A. Ostermann, Y. Wu, F. Yao. {\it A second-order low-regularity integrator for the nonlinear Schr\"odinger equation.}
{https://arxiv.org/abs/2109.01189}.


%
\bibitem{RS}
F.~Rousset, K.~Schratz, \textit{A general framework of low regularity integrators}, to appear in SIAM J. Numer. Anal., {http://arxiv.org/abs/2010.01640}.




\bibitem{physics2}
J. Williams, R. Walser, C. Wieman, J. Cooper, and M. Holland. {\it Achieving steady-state Bose-Einstein condensation.} Physical Review A - Atomic, Molecular, and Optical Physics, 57(3):2030--2036, 1998.

\end{thebibliography}
\end{document}